\documentclass[12pt]{article}

\usepackage[a4paper,margin=1in]{geometry}

\usepackage{amssymb,amsthm,amsmath,mathtools}

\mathtoolsset{showonlyrefs=true}

\numberwithin{equation}{section}
\numberwithin{figure}{section}
\numberwithin{table}{section}

\theoremstyle{definition}

\newtheorem{theorem}{Theorem}[section]
\newtheorem{lemma}[theorem]{Lemma}
\newtheorem{corollary}[theorem]{Corollary}

\newtheorem{fact}[theorem]{Fact}
\newtheorem{proposition}[theorem]{Proposition}

\theoremstyle{definition}
\newtheorem{definition}[theorem]{Definition}
\newtheorem{example}[theorem]{Example}

\theoremstyle{remark}
\newtheorem{remark}[theorem]{Remark}

\usepackage{enumitem}
\setlist[enumerate,1]{label=(\arabic*)}
\setlist[enumerate,2]{label=(\roman*)}

\usepackage{hyperref}
\usepackage[style=alphabetic,isbn=false,url=false,doi=false,giveninits=true,maxnames=9,sorting=nyt,backend=bibtex]{biblatex} 
\renewbibmacro{in:}{}
\bibliography{reference.bib}
\DeclareFieldFormat[article,periodical]{volume}{\mkbibbold{#1}}

\DeclareMathOperator*{\einf}{ess \, inf}
\DeclareMathOperator*{\esup}{ess \, sup}

\title{Borell--Brascamp--Lieb inequality with finitely many output functions}
\author{Takashi Satomi}

\begin{document}

\maketitle

\begin{abstract}

The classical Borell--Brascamp--Lieb inequality for multiple functions is an integral inequality relating finitely many input functions to a single output function.
In this paper, we give an extension that allows a distinct output function for each input function.

More precisely, we establish the following inequality.
For a weight $ \lambda = ( \lambda_1 , \dots , \lambda_m ) $, we set $ z_\lambda ( x ) \coloneqq \sum_{ i = 1 }^m \lambda_i x_i $ and denote by $ M_p^\lambda $ the weighted power mean of power $ p $.
In addition, the power transform $ Q_d $ is defined as the continuous extension of $ p \mapsto p / ( 1 - d p ) $.
We consider integrable functions $ f_1 , \dots , f_m , g_1 , \dots , g_m \colon \mathbb{ R }^d \to \mathbb{ R }_{ \geq 0 } $ satisfying $ 0 < \| f_i \|_1 < \infty $ for every $ i = 1 , 2 , \dots , m $.
Then one has
\begin{align}
\einf_{ x = ( x_1 , \dots , x_m ) }
M_p^\lambda \left( \frac{ g_1 ( z_\lambda ( x ) ) }{ f_1 ( x_1 ) }, \dots , \frac{ g_m ( z_\lambda ( x ) ) }{ f_m ( x_m ) } \right)
\leq M_{ Q_d ( p ) }^\lambda \left( \frac{ \| g_1 \|_1 }{ \| f_1 \|_1 }, \dots , \frac{ \| g_m \|_1 }{ \| f_m \|_1 } \right)
\end{align}
for any $ - \infty \leq p \leq 1 / d $, where the essential infimum is taken over the set of points satisfying $ f_i ( x_i ) > 0 $ for every $ i = 1 , 2 , \dots , m $.

When all output functions $ g_1 , \dots , g_m $ are equal, this recovers the classical Borell--Brascamp--Lieb inequality for multiple functions.
When $ p = 0 $ and the hypothesis is imposed pointwise, it coincides with a special case of the multi-output Pr\'ekopa--Leindler inequality of Cordero-Erausquin--Maurey.

\end{abstract}



\section{Introduction}

Deriving inequalities between integrals from pointwise comparisons involving several functions is a fundamental problem in convex geometry and analysis.
The Borell--Brascamp--Lieb inequality \cite[Theorem 3.1]{MR404559} \cite[Theorem 3.3]{MR450480} is a fundamental result of this type, and its version for multiple functions (Theorem~\ref{thm:multi input BBL}) is an integral inequality relating finitely many input functions $ f_1 , f_2 , \dots , f_m $ to one output function $ g $ through weighted power means.

In contrast, our main theorem (Theorem~\ref{thm:main-BBL}) allows an output function $ g_i $ to be assigned independently to each input function $ f_i $, and hence concerns an integral inequality involving finitely many distinct output functions.
To formulate both the classical Borell--Brascamp--Lieb inequality and our main result, we first introduce weighted power means.

\begin{definition}[Weighted power mean {\cite[Sections 2.2 and 2.3]{MR46395}}]
\label{def:weighted power mean definition}

A tuple $ \lambda = ( \lambda_1 , \lambda_2 , \dots , \lambda_m ) $ of real numbers is called a weight if
\begin{align}
    \lambda_1 , \lambda_2 , \dots , \lambda_m > 0, & &
    \sum_{ i = 1 }^m \lambda_i = 1.
\end{align}
For a weight $ \lambda = ( \lambda_1 , \lambda_2 , \dots , \lambda_m ) $ and $ - \infty \leq p \leq \infty $, we define the weighted power mean $ M_p^\lambda \colon \mathbb{ R }_{ \geq 0 }^m \to \mathbb{ R }_{ \geq 0 } $ by
\begin{align}
  M_p^\lambda ( a_1 , a_2 , \dots , a_m )
  & \coloneqq
  \left\{
  \begin{aligned}
    & \prod_{ i = 1 }^m { a_i }^{ \lambda_i } & & \text{ if $ p = 0 $ }, \\
    & \min \{ a_1 , a_2 , \dots , a_m \} & & \text{ if $ p = - \infty $ }, \\
    & \max \{ a_1 , a_2 , \dots , a_m \} & & \text{ if $ p = \infty $ }, \\
    & 0 & & \text{ if $ - \infty \leq p \leq 0 $ and $ \prod_{ i = 1 }^m a_i = 0 $ }, \\
    & \left( \sum_{ i = 1 }^m \lambda_i { a_i }^p \right)^{ 1 / p } & & \text{ otherwise }
  \end{aligned}
  \right.
\end{align}
for $ a_1 , a_2 , \dots , a_m \geq 0 $.
    
\end{definition}

In Definition~\ref{def:weighted power mean definition}, the cases are consistent where they overlap, and the values for $ p = 0 , \pm \infty $ and for $ \prod_{ i = 1 }^m a_i = 0 $ agree with the continuous extension of $ \left( \sum_{ i = 1 }^m \lambda_i { a_i }^p \right)^{ 1 / p } $ (Lemma~\ref{lem:weighted power mean continuous}).

For a tuple $ \lambda = ( \lambda_1 , \lambda_2 , \dots , \lambda_m ) $ of positive real numbers, we define $ z_\lambda \colon ( \mathbb{ R }^d )^m \to \mathbb{ R }^d $ by
\begin{align}
    z_\lambda ( x_1 , x_2 , \dots , x_m ) \coloneqq \sum_{ i = 1 }^m \lambda_i x_i
\end{align}
for $ x_1 , x_2 , \dots , x_m \in \mathbb{ R }^d $.
For a function $ f \colon \mathbb{ R }^d \to \mathbb{ R }_{ \geq 0 } $, we write
\begin{align}
    S ( f ) \coloneqq \{ x \in \mathbb{ R }^d \mid f ( x ) > 0 \},
\end{align}
and $ \| f \|_p $ for the $ L^p $ norm of $ f $ with respect to Lebesgue measure $ dx $ on $ \mathbb{ R }^d $.
Then the following theorem holds.

\begin{theorem}[Borell--Brascamp--Lieb inequality for multiple functions]
    \label{thm:multi input BBL}

Let
\begin{align}
    f_1 , f_2 , \dots , f_m , g \colon \mathbb{ R }^d \to \mathbb{ R }_{ \geq 0 }
\end{align}
be integrable functions satisfying $ 0 < \| f_i \|_1 < \infty $ for every $ i = 1 , 2 , \dots , m $, and set
\begin{align}
    S = S ( f_1 ) \times S ( f_2 ) \times \dots \times S ( f_m ). \label{eq:multi input BBL S definition}
\end{align}
For any weight $ \lambda = ( \lambda_1 , \lambda_2 , \dots , \lambda_m ) $ and any $ - 1 / d \leq \alpha \leq \infty $, if
    \begin{align}
        M_\alpha^\lambda ( f_1 ( x_1 ) , \dots , f_m ( x_m ) )
        \leq g \left( z_\lambda ( x ) \right) & & \text{ a.e. $ x = ( x_1 , \dots , x_m ) \in S $ } \label{eq:multi input BBL assumption}
    \end{align}
with respect to Lebesgue measure on $ S $, then
    \begin{align}
        M_{ \alpha / (1 + d \alpha ) }^\lambda ( \| f_1 \|_1 , \| f_2 \|_1 , \dots , \| f_m \|_1 ) 
        \leq \| g \|_1. \label{eq:multi input BBL claim}
    \end{align}
Here, for $ \alpha = - 1 / d , \infty $, we set $ \alpha / (1 + d \alpha ) = - \infty , 1 / d $, respectively.
    
\end{theorem}

There are several known formulations of Theorem~\ref{thm:multi input BBL}, and the above formulation differs slightly from those of Borell and Brascamp--Lieb.
The formulation above is the multiple-function version obtained from the two-function version of Ishige--Liu--Salani \cite[Theorem 1.1]{MR4958493}.
The multiple-function version follows inductively from the two-function case by normalizing the weights (Section~\ref{sec:original BBL}).

The case $ \alpha = 0 $ of Theorem~\ref{thm:multi input BBL} is called the Pr\'ekopa--Leindler inequality.
On the other hand, by taking $ \alpha = \infty $ and taking the input and output functions to be indicator functions, one obtains the weighted Brunn--Minkowski inequality.
The relations among these inequalities and Theorem~\ref{thm:multi input BBL} are discussed in detail by Gardner \cite[Section 10]{MR1898210}.

In Theorem~\ref{thm:multi input BBL}, we put
\begin{align}
    p = - \alpha, & & x = ( x_1 , x_2 , \dots , x_m ) \in S, & & z = z_\lambda ( x ).
\end{align}
Then
\begin{align}
    \frac{ g ( z ) }{ M_\alpha^\lambda ( f_1 ( x_1 ) , f_2 ( x_2 ) , \dots , f_m ( x_m ) ) }
    = M_p^\lambda \left( \frac{ g ( z ) }{ f_1 ( x_1 ) } , \frac{ g ( z ) }{ f_2 ( x_2 ) }, \dots , \frac{ g ( z ) }{ f_m ( x_m ) } \right) \label{eq:same output}
\end{align}
holds (Section~\ref{subsec:original BBL same output}).
Thus, the hypothesis in \eqref{eq:multi input BBL assumption} can be written as
\begin{align}
    \einf_{ x = ( x_1 , x_2 , \dots , x_m ) \in S } M_p^\lambda \left( \frac{ g ( z_\lambda ( x ) ) }{ f_1 ( x_1 ) } , \frac{ g ( z_\lambda ( x ) ) }{ f_2 ( x_2 ) }, \dots , \frac{ g ( z_\lambda ( x ) ) }{ f_m ( x_m ) } \right)
     \geq 1,
\end{align}
where the essential infimum is taken with respect to Lebesgue measure on $ S $.
In this paper, we replace the common output function $ g $ by output functions $ g_i $ depending on the index $ i $.
More precisely, in addition to the $ m $ input functions $ f_1 , f_2 , \dots , f_m $, we consider $ m $ output functions $ g_1 , g_2 , \dots , g_m $, and we define
\begin{align}
L ( x ) \coloneqq M_p^\lambda \left( \frac{ g_1 ( z_\lambda ( x ) ) }{ f_1 ( x_1 ) } , \frac{ g_2 ( z_\lambda ( x ) ) }{ f_2 ( x_2 ) }, \dots , \frac{ g_m ( z_\lambda ( x ) ) }{ f_m ( x_m ) } \right), & &
K \coloneqq \einf_{ x \in S } L ( x ) \label{eq:main theorem essential inf}
\end{align}
for $ x = ( x_1 , x_2 , \dots , x_m ) \in S $.
Unlike the common-output case, in general $ L ( x ) $ cannot be written in a simple separated form corresponding to \eqref{eq:same output}.
Thus, one cannot directly reduce the problem to Theorem~\ref{thm:multi input BBL}.
We therefore establish a multi-output Borell--Brascamp--Lieb inequality that estimates $ K $ by the weighted power mean of the ratios $ \| g_i \|_1 / \| f_i \|_1 $.

We define the power transform $ Q_d \colon [ - \infty , 1 / d ] \to [ - \infty , \infty ] $ by
\begin{align}
    Q_d ( p ) \coloneqq 
  \left\{
  \begin{aligned}
    &   \frac{ p }{ 1 - d p } & & \text{ if $ - \infty < p < \dfrac{ 1 }{ d } $ } \\
    & - \frac{ 1 }{ d } & & \text{ if $ p = - \infty $ } \\
            & \infty & & \text{ if $ p = \dfrac{ 1 }{ d } $ }
  \end{aligned}
  \right. . \label{eq:power transform}
\end{align}
Then our main theorem is as follows.

\begin{theorem}
\label{thm:main-BBL}

Let $ d $, $ f_1 , f_2 , \dots , f_m $, $ S $, and $ \lambda $ be as in Theorem~\ref{thm:multi input BBL}.
For any $ - \infty \leq p \leq 1 / d $ and any integrable functions $ g_1 , g_2 , \dots , g_m \colon \mathbb{ R }^d \to \mathbb{ R }_{ \geq 0 } $, we define $ L $ and $ K $ by \eqref{eq:main theorem essential inf}.
Then
\begin{align}
K
\leq M_{ Q_d ( p ) }^\lambda \left( \frac{ \| g_1 \|_1 }{ \| f_1 \|_1 } , \frac{ \| g_2 \|_1 }{ \| f_2 \|_1 } , \dots , \frac{ \| g_m \|_1 }{ \| f_m \|_1 } \right). \label{eq:main-BBL-claim}
\end{align}

\end{theorem}

When $ g_1 = g_2 = \dots = g_m $, Theorem~\ref{thm:main-BBL} is equivalent to Theorem~\ref{thm:multi input BBL}.
Moreover, when the output functions $ g_1 , g_2 , \dots , g_m $ are mutually proportional, the problem reduces to the common-output case (Section~\ref{subsec:original BBL same output}).
For $ - \infty \leq p \leq 1 / ( d + 1 ) $, Theorem~\ref{thm:main-BBL} can be derived from Theorem~\ref{thm:multi input BBL} by combining the H\"older inequality (Lemma~\ref{lem:Holder}) with an integral inequality (Lemma~\ref{lem:integral inequality}).
However, for $ 1 / ( d + 1 ) < p \leq 1 / d $, the power falls outside the range in which the integral inequality is valid, so this method does not reduce Theorem~\ref{thm:main-BBL} to Theorem~\ref{thm:multi input BBL} (Section~\ref{subsec:known inequalities low power}).

When $ p = 0 $ and the hypothesis is imposed pointwise, Theorem~\ref{thm:main-BBL} is a special case of the multi-output Pr\'ekopa--Leindler inequality of Cordero-Erausquin--Maurey (Fact~\ref{fact:Cordero-Erausquin--Maurey}).

Theorem~\ref{thm:main-BBL} is sharp in the following sense.
When we fix $ d , p , \lambda $, the least constant $ C $ such that
\begin{align}
K
\leq C M_{ Q_d ( p ) }^\lambda \left( \frac{ \| g_1 \|_1 }{ \| f_1 \|_1 } , \dots , \frac{ \| g_m \|_1 }{ \| f_m \|_1 } \right) \label{eq:product best}
\end{align}
holds for any family of functions is $ C = 1 $.
We call $ C $ the multiplicative constant (Section~\ref{subsec:equality product constant best}).
Furthermore, when $ m \geq 2 $, the power $ Q_d ( p ) $ cannot be replaced by any smaller power while keeping the multiplicative constant equal to $ 1 $ (Section~\ref{subsec:equality power best}).

A change of variables also yields an inequality in which the evaluation point $ z_\lambda ( x ) $ in Theorem~\ref{thm:main-BBL} is replaced by a general linear combination $ z_\beta ( x ) $ (Corollary~\ref{cor:general-linear-coefficients}).

We now describe the strategy of the proof of Theorem~\ref{thm:main-BBL}.
Since the reduction to Theorem~\ref{thm:multi input BBL} described above is unavailable when $ 1 / ( d + 1 ) < p \leq 1 / d $, we prove the theorem independently of that reduction.
The proof consists of two stages: the one-dimensional case and the higher-dimensional case.

In dimension one, we first prove a minimum-type integral inequality for multiple functions (Lemma~\ref{lem:multifunction-minimum-integral}) by iterating the two-function minimum-type integral inequality of Brascamp--Lieb (Fact~\ref{fact:Brascamp--Lieb}), and then combine it with inequalities for weighted power means.
In higher dimensions, we apply lower-dimensional results successively to slice integrals and use the composition formula $ Q_e ( Q_d ( p ) ) = Q_{ d + e } ( p ) $ for the power transform (Lemma~\ref{lem:Q_d compose}).

The paper is organized as follows.
Section~\ref{sec:weighted mean prepare} collects properties of weighted power means needed later.
Section~\ref{sec:original BBL} explains the relation between Theorem~\ref{thm:main-BBL} and previous results.
In Section~\ref{sec:one dimension}, we derive the minimum-type integral inequality for multiple functions (Lemma~\ref{lem:multifunction-minimum-integral}) from the two-function minimum-type integral inequality of Brascamp--Lieb (Fact~\ref{fact:Brascamp--Lieb}) and use it to prove Theorem~\ref{thm:main-BBL} in dimension one.
In Section~\ref{sec:high dimension}, we use slice integrals and induction on the dimension to deduce the general case from the one-dimensional case.
In Section~\ref{sec:equality}, we construct families of functions for which equality in Theorem~\ref{thm:main-BBL} holds in the case of proportional outputs.
These equality examples show that the multiplicative constant $ C = 1 $ in Theorem~\ref{thm:main-BBL} is optimal and that the power $ Q_d ( p ) $ cannot be replaced by a smaller one (Proposition~\ref{prop:power best}).
Finally, Section~\ref{sec:general-linear-coefficients} derives an inequality in which the evaluation point in Theorem~\ref{thm:main-BBL} is replaced by a general linear combination (Corollary~\ref{cor:general-linear-coefficients}), and explains that the same optimality statements for the multiplicative constant and power remain valid.

\section{Weighted power means and preliminaries}
\label{sec:weighted mean prepare}

In this section, we collect basic properties of the weighted power mean $ M_p^\lambda $ that will be used later.
In Section~\ref{subsec:weighted mean prepare basic formulas}, we collect as basic properties the results stated by Hardy--Littlewood--Pólya \cite{MR46395}, and we further prove joint continuity with respect to the power and the vector (Lemma~\ref{lem:weighted power mean continuous}).
In Section~\ref{subsec:weighted mean prepare power transform}, we prove continuity of the dimension-dependent power transform $ Q_d $ and its composition formula.
In Section~\ref{subsec:weighted mean prepare Holder inequality}, we prove the H\"older inequality relating the powers $ p $, $ 1 / d $, and $ Q_d ( p ) $.
Finally, in Section~\ref{subsec:weighted mean prepare integral inequality}, we derive the reverse Minkowski integral inequality on Euclidean space.

\subsection{Basic properties of weighted power means}
\label{subsec:weighted mean prepare basic formulas}

In this subsection, we review basic properties of the weighted power mean $ M_p^\lambda $ stated by Hardy--Littlewood--Pólya \cite{MR46395}.
For $ a = ( a_1 , a_2 , \dots , a_m ) $, we henceforth simply write $ M_p^\lambda ( a ) $ for $ M_p^\lambda ( a_1 , a_2 , \dots , a_m ) $.

\begin{fact}[Basic properties of weighted power means]
    \label{fact:weighted power mean properties}

    Let $ \lambda = ( \lambda_1 , \lambda_2 , \dots , \lambda_m ) $ be a weight, and let
\begin{align}
    - \infty \leq p \leq \infty, & &
    a = ( a_1 , a_2 , \dots , a_m ) \in \mathbb{ R }_{ \geq 0 }^m.
\end{align}

\begin{enumerate}
    \item \label{item:weighted power mean properties inverse}
    (Inverse formula \cite[(2.2.9) and Section 2.3]{MR46395})
    If $ a_i > 0 $ for every $ i = 1 , 2 , \dots , m $, then
    \begin{align}
        M_p^\lambda \left( \frac{ 1 }{ a_1 } , \frac{ 1 }{ a_2 } , \dots , \frac{ 1 }{ a_m } \right)
        = \frac{ 1 }{ M_{ - p }^\lambda ( a ) }.
    \end{align}
    
    \item \label{item:weighted power mean properties homogeneity}
    (Positive homogeneity of degree one \cite[(2.2.13)]{MR46395})
    For any $ t \geq 0 $, one has
    \begin{align}
        M_p^\lambda ( t a ) = t M_p^\lambda ( a ).
    \end{align}
    
    \item \label{item:weighted power mean properties element monotone}
    (Componentwise monotonicity \cite[(2.2.15)]{MR46395})
    For any $ b = ( b_1 , b_2 , \dots , b_m ) \in \mathbb{ R }_{ \geq 0 }^m $ satisfying $ a_i \leq b_i $ for every $ i = 1 , 2 , \dots , m $, one has
    \begin{align}
        M_p^\lambda ( a ) \leq M_p^\lambda ( b ).
    \end{align}
    
    \item \label{item:weighted power mean properties power monotone main}
    (Monotonicity in the power \cite[Theorem 16]{MR46395})
    For any $ p \leq q \leq \infty $, one has
    \begin{align}
        M_p^\lambda ( a ) \leq M_q^\lambda ( a ).
    \end{align}
    
    \item \label{item:weighted power mean properties power monotone strict}
    (Monotonicity in the power \cite[Theorem 16]{MR46395})
    Suppose $ a $ is not a constant vector and that $ a_i > 0 $ for every $ i = 1 , 2 , \dots , m $.
    Then for any $ p < q \leq \infty $, one has
    \begin{align}
        M_p^\lambda ( a ) < M_q^\lambda ( a ).
    \end{align}
    
\end{enumerate}

\end{fact}

Fact~\ref{fact:weighted power mean properties} yields the following example, which will be used in Section~\ref{subsec:original BBL same output}.

\begin{example}
    \label{ex:homogeneity-inverse-combine}

    By Fact~\ref{fact:weighted power mean properties} \ref{item:weighted power mean properties inverse} and \ref{item:weighted power mean properties homogeneity}, for any weight $ \lambda = ( \lambda_1 , \lambda_2 , \dots , \lambda_m ) $ and
    \begin{align}
        - \infty \leq p \leq \infty, & &
        t \geq 0 , & &
        a = ( a_1 , a_2 , \dots , a_m ) \in \mathbb{ R }_{ > 0 }^m,
    \end{align}
    one has
    \begin{align}
        M_p^\lambda \left( \frac{ t }{ a_1 } , \frac{ t }{ a_2 } , \dots , \frac{ t }{ a_m } \right)
        = t M_p^\lambda \left( \frac{ 1 }{ a_1 } , \frac{ 1 }{ a_2 } , \dots , \frac{ 1 }{ a_m } \right)
        = \frac{ t }{ M_{ - p }^\lambda ( a ) }.
    \end{align}
    
\end{example}

We equip the extended real interval $ [ - \infty , \infty ] $ with the usual order topology.
The weighted power mean $ M_p^\lambda ( a ) $ has the following joint continuity with respect to $ p $ and $ a $.

\begin{lemma}[Joint continuity]
   \label{lem:weighted power mean continuous}

   For any weight $ \lambda = ( \lambda_1 , \lambda_2 , \dots , \lambda_m ) $, the map
   \begin{align}
       [ - \infty , \infty ] \times \mathbb{ R }_{ \geq 0 }^m \to \mathbb{ R }_{ \geq 0 }, & & ( p , a ) \mapsto M_p^\lambda ( a )
   \end{align}
   is continuous.
   
\end{lemma}

\begin{proof}

First, let $ \nu = 1 , 2 , \dots $ and suppose that
\begin{align}
    \lim_{ \nu \to \infty } a^{ ( \nu ) } = a.
\end{align}
We show that
\begin{align}
    \lim_{ \nu \to \infty } M_p^\lambda ( a^{ ( \nu ) } )
    = M_p^\lambda ( a ) \label{eq:weighted power mean continuous proof a-continuous}
\end{align}
for any $ - \infty \leq p \leq \infty $.
If $ - \infty < p < 0 $ and there exists $ \ell = 1 , 2 , \dots , m $ such that $ a_\ell = 0 $, then
\begin{align}
    0 \leq M_p^\lambda ( a^{ ( \nu ) } )
    \leq \lambda_\ell^{ 1 / p } a_\ell^{ ( \nu ) }
    \to 0.
\end{align}
Thus, we have \eqref{eq:weighted power mean continuous proof a-continuous}.
In all other cases, we get \eqref{eq:weighted power mean continuous proof a-continuous} by the definition.

It remains to show that if
\begin{align}
    \lim_{ \nu \to \infty } p_\nu = p , & &
    \lim_{ \nu \to \infty } a^{ ( \nu ) } = a,
\end{align}
then
\begin{align}
    \lim_{ \nu \to \infty } M_{ p_\nu }^\lambda ( a^{ ( \nu ) } )
    = M_p^\lambda ( a ). \label{eq:weighted power mean continuous proof claim}
\end{align}

First suppose that $ - \infty < p < \infty $.
For any $ \delta > 0 $, we have $ p - \delta < p_\nu < p + \delta $ for all sufficiently large $ \nu $.
Thus, Fact~\ref{fact:weighted power mean properties} \ref{item:weighted power mean properties power monotone main} gives
\begin{align}
    M_{ p - \delta }^\lambda ( a^{ ( \nu ) } )
    \leq M_{ p_\nu }^\lambda ( a^{ ( \nu ) } )
    \leq M_{ p + \delta }^\lambda ( a^{ ( \nu ) } ).
\end{align}
By letting $ \nu \to \infty $ and using \eqref{eq:weighted power mean continuous proof a-continuous}, we obtain
\begin{align}
    M_{ p - \delta }^\lambda ( a )
    \leq \liminf_{ \nu \to \infty }
    M_{ p_\nu }^\lambda ( a^{ ( \nu ) } )
    \leq \limsup_{ \nu \to \infty }
    M_{ p_\nu }^\lambda ( a^{ ( \nu ) } )
    \leq M_{ p + \delta }^\lambda ( a ).
\end{align}
By Hardy--Littlewood--Pólya \cite[Section 2.3, Theorem 3]{MR46395}, we have
\begin{align}
    \lim_{ \delta \downarrow 0 } M_{ p + \delta }^\lambda ( a )
    = \lim_{ \delta \downarrow 0 } M_{ p - \delta }^\lambda ( a )
    = M_p^\lambda ( a ).
\end{align}
Thus, \eqref{eq:weighted power mean continuous proof claim} follows.

Suppose next that $ p = \infty $.
For any $ J > 0 $, we have $ p_\nu > J $ for all sufficiently large $ \nu $.
Thus, we get
\begin{align}
    M_J^\lambda ( a^{ ( \nu ) } )
    \leq M_{ p_\nu }^\lambda ( a^{ ( \nu ) } )
    \leq M_\infty^\lambda ( a^{ ( \nu ) } ).
\end{align}
We have
\begin{align}
    M_J^\lambda ( a )
    \leq \liminf_{ \nu \to \infty } M_{ p_\nu }^\lambda ( a^{ ( \nu ) } )
    \leq \limsup_{ \nu \to \infty } M_{ p_\nu }^\lambda ( a^{ ( \nu ) } )
    \leq M_\infty^\lambda ( a )
\end{align}
by letting $ \nu \to \infty $ and using \eqref{eq:weighted power mean continuous proof a-continuous}, and hence
\begin{align}
    \lim_{ J \to \infty } M_J^\lambda ( a )
    = M_\infty^\lambda ( a )
\end{align}
holds by Hardy--Littlewood--Pólya \cite[Section 2.3, Theorem 4]{MR46395}.
Thus, we obtain \eqref{eq:weighted power mean continuous proof claim} by letting $ J \to \infty $.

In the case of $ p = - \infty $, we have
\begin{align}
    M_{ - \infty }^\lambda ( a^{ ( \nu ) } )
    \leq M_{ p_\nu }^\lambda ( a^{ ( \nu ) } )
    \leq M_{ - J }^\lambda ( a^{ ( \nu ) } ).
\end{align}
Thus, we obtain \eqref{eq:weighted power mean continuous proof claim} by a similar argument.
\end{proof}

\subsection{Dimension-dependent power transform}
\label{subsec:weighted mean prepare power transform}

In this subsection, we study properties of the power transform $ Q_d ( p ) $ depending on the dimension $ d $.
Although Theorem~\ref{thm:main-BBL} uses only $ d = 1 , 2 , \dots $, the power transform $ Q_d ( p ) $ can be defined by \eqref{eq:power transform} for any real number $ d > 0 $.

\begin{lemma}[Continuity of the power transform]
    \label{lem:Q_d continuous}

    For any $ d > 0 $, the map $ Q_d \colon [ - \infty , 1 / d ] \to [ - \infty , \infty ] $ is continuous.
\end{lemma}

\begin{proof}

For $ - \infty < p < 1 / d $, this follows from the definition.
Moreover, we have
\begin{align}
    \lim_{ p \to - \infty } Q_d ( p )
    & = \lim_{ p \to - \infty } \frac{ p }{ 1 - d p }
    = - \frac{ 1 }{ d }
    = Q_d ( - \infty ), \\
    \lim_{ p \uparrow 1 / d } Q_d ( p )
    & = \lim_{ p \uparrow 1 / d } \frac{ p }{ 1 - d p }
    = \infty
    = Q_d \left( \frac{ 1 }{ d } \right).
\end{align}
Thus, $ Q_d $ is continuous.
\end{proof}

The power transform $ Q_d $ satisfies the following equivalence.

\begin{lemma}
    \label{lem:applicability}

    For any
    \begin{align}
        d , e > 0, & &
        - \infty \leq p \leq \frac{ 1 }{ d },
    \end{align}
    one has
    \begin{align}
        Q_d ( p ) \leq \frac{ 1 }{ e }
        \Longleftrightarrow
        p \leq \frac{ 1 }{ d + e }.
    \end{align}
    
\end{lemma}

\begin{proof}

Both sides hold for $ p = - \infty $, whereas neither side holds for $ p = 1 / d $.
If $ - \infty < p < 1 / d $, then $ 1 - d p > 0 $, and hence
\begin{align}
    Q_d ( p ) \leq \frac{ 1 }{ e }
    \Longleftrightarrow \frac{ p }{ 1 - d p } \leq \frac{ 1 }{ e }
    \Longleftrightarrow e p \leq 1 - d p
    \Longleftrightarrow ( d + e ) p \leq 1
    \Longleftrightarrow p \leq \frac{ 1 }{ d + e }
\end{align}
holds.
\end{proof}

The power transforms also satisfy the following composition formula.

\begin{lemma}[Composition formula for the power transform]
    \label{lem:Q_d compose}

    For any real numbers $ d , e > 0 $ and any $ - \infty \leq p \leq 1 / ( d + e ) $, the quantity $ Q_e ( Q_d ( p ) ) $ is well-defined, that is, $ p \leq 1 / d $ and $ Q_d ( p ) \leq 1 / e $, and
    \begin{align}
        Q_e ( Q_d ( p ) ) = Q_{ d + e } ( p ). \label{eq:Q_d compose claim}
    \end{align}
    
\end{lemma}

\begin{proof}

The inequalities $ p \leq 1 / d $ and $ Q_d ( p ) \leq 1 / e $ follow from
\begin{align}
    p \leq \frac{ 1 }{ d + e }
    < \frac{ 1 }{ d }
\end{align}
and Lemma~\ref{lem:applicability}.
Thus, the quantity $ Q_e ( Q_d ( p ) ) $ is well-defined.

If $ - \infty < p < 1 / ( d + e ) $, then
\begin{align}
    Q_e ( Q_d ( p ) )
    = Q_e \left( \frac{ p }{ 1 - d p } \right)
    = \frac{ p / ( 1 - d p ) }{ 1 - e p / ( 1 - d p ) }
    = \frac{ p }{ 1 - ( d + e ) p }
    = Q_{ d + e } ( p ).
\end{align}
By Lemma~\ref{lem:Q_d continuous}, the equality \eqref{eq:Q_d compose claim} also holds for $ p = - \infty $ and $ p = 1 / ( d + e ) $.
\end{proof}

Lemma~\ref{lem:Q_d compose} will be used in the proof of the dimension-additivity lemma (Lemma~\ref{lem:dimension-additivity}) in Section~\ref{sec:high dimension} to combine the powers that arise there.

\subsection{H\"older inequality for weighted power means}
\label{subsec:weighted mean prepare Holder inequality}

In this subsection, we prove the H\"older inequality for $ M_p^\lambda $ in the range $ - \infty \leq p \leq 1 / d $.
Since $ Q_d ( p ) $ is well-defined in this range, we have the following inequality.

\begin{lemma}[H\"older inequality]
    \label{lem:Holder}

    For any weight $ \lambda = ( \lambda_1 , \lambda_2 , \dots , \lambda_m ) $ and
    \begin{align}
        d > 0, & &
        - \infty \leq p \leq \frac{ 1 }{ d }, & &
        a = ( a_1 , a_2 , \dots , a_m ) \in \mathbb{ R }_{ \geq 0 }^m , & &
        b = ( b_1 , b_2 , \dots , b_m ) \in \mathbb{ R }_{ \geq 0 }^m,
    \end{align}
    one has
    \begin{align}
        M_p^\lambda ( a_1 b_1 , a_2 b_2 , \dots , a_m b_m )
        \leq M_{ 1 / d }^\lambda ( a ) M_{ Q_d ( p ) }^\lambda ( b ). \label{eq:Holder claim}
    \end{align}
    
\end{lemma}

\begin{proof}

If $ p \in ( - \infty , 1 / d ) \setminus \{ 0 \} $, then
\begin{align}
    \frac{ 1 }{ p }
    = d + \left( \frac{ 1 }{ p } - d \right)
    = d + \frac{ 1 - d p }{ p }
    = \frac{ 1 }{ 1 / d } + \frac{ 1 }{ Q_d ( p ) }.
\end{align}
Thus, \eqref{eq:Holder claim} follows from Hardy--Littlewood--Pólya \cite[Section 2.8, Theorem 13]{MR46395} and Lemma~\ref{lem:weighted power mean continuous}.
The remaining cases follow from Lemmas~\ref{lem:weighted power mean continuous} and \ref{lem:Q_d continuous}.
\end{proof}

We also discuss the following reverse H\"older-type inequality, which will be needed in the proof of the one-dimensional case of the main theorem (Section~\ref{subsec:one dimension proof}).

\begin{lemma}
    \label{lem:Holder infinity}

    Let $ \lambda $, $ a $, and $ b $ be as in Lemma~\ref{lem:Holder}.
    Then one has
    \begin{align}
        M_p^\lambda ( a ) M_{ - \infty }^\lambda ( b )
        \leq M_p^\lambda ( a_1 b_1 , a_2 b_2 , \dots , a_m b_m )
    \end{align}
    for any $ - \infty \leq p \leq \infty $.
    
\end{lemma}

\begin{proof}
    
For any $ i = 1 , 2 , \dots , m $, we have $ b_i \geq M_{ - \infty }^\lambda ( b ) $.
Thus, the assertion follows from Fact~\ref{fact:weighted power mean properties} \ref{item:weighted power mean properties homogeneity} and \ref{item:weighted power mean properties element monotone}.
\end{proof}

\subsection{Reverse Minkowski integral inequality for weighted power means}
\label{subsec:weighted mean prepare integral inequality}

In this subsection, we give the following reverse Minkowski integral inequality for the weighted power mean $ M_p^\lambda $.

\begin{lemma}[Reverse Minkowski integral inequality for weighted power means]
\label{lem:integral inequality}

    Let $ \lambda = ( \lambda_1 , \lambda_2 , \dots , \lambda_m ) $ be a weight and let $ - \infty \leq p \leq 1 $.
    For any integrable functions $ h_1 , h_2 , \dots , h_m \colon \mathbb{ R }^d \to \mathbb{ R }_{ \geq 0 } $, we set
    \begin{align}
        h ( x ) \coloneqq ( h_1 ( x ) , h_2 ( x ) , \dots , h_m ( x ) ), & &
        \Gamma ( x ) \coloneqq M_p^\lambda ( h ( x ) ).
    \end{align}
    Then $ \Gamma $ is integrable and
    \begin{align}
        \| \Gamma \|_1
        \leq M_p^\lambda \left( \| h_1\|_1 , \| h_2\|_1 , \dots , \| h_m\|_1 \right).
    \end{align}
    
\end{lemma}

\begin{proof}

If $ p \in ( - \infty , 1 ) \setminus \{ 0 \} $, the assertion follows from Hardy--Littlewood--Pólya \cite[Section 6.13, Theorem 198]{MR46395}.
In the remaining cases, Fact~\ref{fact:weighted power mean properties} \ref{item:weighted power mean properties power monotone main} gives
\begin{align}
    0 \leq \Gamma ( x ) \leq M_1^\lambda ( h ( x ) ), & &
    \int_{ \mathbb{ R }^d } M_1^\lambda ( h ( x ) ) \; dx
    = \sum_{ i = 1 }^m \lambda_i \| h_i \|_1
    < \infty.
\end{align}
Thus, the assertion follows from the dominated convergence theorem and Lemma~\ref{lem:weighted power mean continuous}.
\end{proof}

Lemma~\ref{lem:integral inequality} will be used both to derive Theorem~\ref{thm:main-BBL} from Theorem~\ref{thm:multi input BBL} in the low-power range $ - \infty \leq p \leq 1 / ( d + 1 ) $ (Section~\ref{subsec:known inequalities low power}) and to prove the one-dimensional case of Theorem~\ref{thm:main-BBL} (Section~\ref{subsec:one dimension proof}).

\begin{remark}
    \label{rem:integral inequality sharpest}

    When $ m \geq 2 $, Lemma~\ref{lem:integral inequality} holds only in the range $ - \infty \leq p \leq 1 $ and fails for $ 1 < p \leq \infty $.
    Indeed, for $ 1 < p \leq \infty $, a counterexample can be constructed as follows.

    Let $ E_1 , E_2 , \dots , E_m \subset \mathbb{ R }^d $ be pairwise disjoint measurable sets of measure $ 1 $, and let $ h_i = \chi_{ E_i } $ be the indicator function of $ E_i $ for $ i = 1 , 2 , \dots , m $.
    Then
    \begin{align}
        M_p^\lambda \left( \| h_1\|_1 , \| h_2\|_1 , \dots , \| h_m\|_1 \right)
        = M_p^\lambda ( 1 , 1 , \dots , 1 )
        = 1.
    \end{align}
    On the other hand, if $ 1 < p < \infty $, then
    \begin{align}
        \| \Gamma \|_1
        = \sum_{ i = 1 }^m \lambda_i^{ 1 / p }
        > \sum_{ i = 1 }^m \lambda_i
        = 1,
    \end{align}
    while if $ p = \infty $, then
    \begin{align}
        \| \Gamma \|_1
        = m
        > 1.
    \end{align}
    Thus, these functions give a counterexample to Lemma~\ref{lem:integral inequality}.

    This counterexample shows that Lemma~\ref{lem:integral inequality} fails for $ 1 < p \leq \infty $.
    This is the reason why the derivation of Theorem~\ref{thm:main-BBL} from Theorem~\ref{thm:multi input BBL} in the low-power range does not extend directly to the higher-power range (Remark~\ref{rem:high power}).
    
\end{remark}

\section{Relation between the main theorem and previous results}
\label{sec:original BBL}

In this section, we explain the relations of Theorems~\ref{thm:multi input BBL} and \ref{thm:main-BBL} to previous results.

Ishige--Liu--Salani \cite[Theorem 1.1]{MR4958493} state the case $ m = 2 $ of Theorem~\ref{thm:multi input BBL} and point out immediately afterward that it can be extended to $ m = 3 , 4 , \dots $.
Indeed, the case $ m = 3 , 4 , \dots $ follows by normalizing the weights and iterating the case $ m = 2 $.
The definition of the weighted power mean in Ishige--Liu--Salani differs from that of $ M_p^\lambda ( a ) $ in Definition~\ref{def:weighted power mean definition}.
In their definition, the power mean is equal to $ 0 $ for any power $ p $ whenever the vector $ a $ has at least one zero component.
Theorem~\ref{thm:multi input BBL} uses the weighted power mean $ M_p^\lambda $ of Definition~\ref{def:weighted power mean definition}.
However, because the hypothesis is imposed on the product $ S $ of the positivity sets $ S ( f_i ) $ of the input functions, the two hypotheses are equivalent.

Various aspects of Theorem~\ref{thm:multi input BBL}, including equality cases and stability, have been studied in \cite{MR444863} \cite{MR2678033} \cite{MR2787578} \cite{MR3235317} \cite{MR3684803} \cite{MR3645132} \cite{MR3866904} \cite{MR3977216} \cite{MR4266751} \cite{MR4740622} \cite{figalli2025sharp} \cite{MR4950459} \cite{van2025brunn}.
For further details, see Ishige--Liu--Salani \cite{MR4958493}.

In Section~\ref{subsec:original BBL same output}, we show that the common-output case of Theorem~\ref{thm:main-BBL} is equivalent to Theorem~\ref{thm:multi input BBL}, and that the proportional-output case reduces to the common-output case.
In Section~\ref{subsec:known inequalities low power}, we explain that Theorem~\ref{thm:main-BBL} can be derived from Theorem~\ref{thm:multi input BBL} in the low-power range $ - \infty \leq p \leq 1 / ( d + 1 ) $, whereas the same derivation fails in the high-power range $ 1 / ( d + 1 ) < p \leq 1 / d $.
In Section~\ref{subsec:known inequalities Cordero-Erausquin--Maurey}, we explain the relation between the generalized Pr\'ekopa--Leindler inequality of Cordero-Erausquin--Maurey (Fact~\ref{fact:Cordero-Erausquin--Maurey}) and Theorem~\ref{thm:main-BBL}.

\subsection{Common and proportional outputs}
\label{subsec:original BBL same output}

In this subsection, we show that Theorem~\ref{thm:main-BBL} in the common-output case $ g_1 = g_2 = \dots = g_m $ is equivalent to Theorem~\ref{thm:multi input BBL}.
We also show that the proportional-output case reduces to the common-output case.

We first prepare several observations for proving the equivalence between the common-output case of Theorem~\ref{thm:main-BBL} and Theorem~\ref{thm:multi input BBL}.
Let $ d $, $ f_1 , f_2 , \dots , f_m $, $ S $, $ \lambda $, $ g_1 , g_2 , \dots , g_m $, $ L $, and $ K $ be as in Theorem~\ref{thm:main-BBL}, and suppose
\begin{align}
    p = - \alpha , & & g = g_1 = g_2 = \dots = g_m.
\end{align}
We first show that
\begin{align}
    K
    = \einf_{ x = ( x_1 , x_2 , \dots , x_m ) \in S } \frac{ g ( z_\lambda ( x ) ) }{ M_\alpha^\lambda ( f_1 ( x_1 ) , f_2 ( x_2 ) , \dots , f_m ( x_m ) ) }. \label{eq:K transform}
\end{align}
For any $ x = ( x_1 , x_2 , \dots , x_m ) \in S $, Example~\ref{ex:homogeneity-inverse-combine} gives
\begin{align}
L ( x )
& = M_p^\lambda \left( \frac{ g_1 ( z_\lambda ( x ) ) }{ f_1 ( x_1 ) } , \frac{ g_2 ( z_\lambda ( x ) ) }{ f_2 ( x_2 ) }, \dots , \frac{ g_m ( z_\lambda ( x ) ) }{ f_m ( x_m ) } \right) \\
& = M_p^\lambda \left( \frac{ g ( z_\lambda ( x ) ) }{ f_1 ( x_1 ) } , \frac{ g ( z_\lambda ( x ) ) }{ f_2 ( x_2 ) }, \dots , \frac{ g ( z_\lambda ( x ) ) }{ f_m ( x_m ) } \right) \\
& = \frac{ g ( z_\lambda ( x ) ) }{ M_\alpha^\lambda ( f_1 ( x_1 ) , f_2 ( x_2 ) , \dots , f_m ( x_m ) ) }.
\end{align}
Thus, \eqref{eq:K transform} holds.

Next, we show that
\begin{align}
M_{ Q_d ( p ) }^\lambda \left( \frac{ \| g \|_1 }{ \| f_1 \|_1 } , \frac{ \| g \|_1 }{ \| f_2 \|_1 } , \dots , \frac{ \| g \|_1 }{ \| f_m \|_1 } \right)
= \frac{ \| g \|_1 } { M_{ \alpha / ( 1 + d \alpha ) }^\lambda ( \| f_1 \|_1 , \| f_2 \|_1 , \dots , \| f_m \|_1 ) }. \label{eq:norm same output}
\end{align}
By Example~\ref{ex:homogeneity-inverse-combine}, we have
\begin{align}
M_{ Q_d ( p ) }^\lambda \left(
    \frac{ \| g \|_1 }{ \| f_1 \|_1 } ,
    \frac{ \| g \|_1 }{ \| f_2 \|_1 } ,
    \dots ,
    \frac{ \| g \|_1 }{ \| f_m \|_1 }
\right)
= \frac{ \| g \|_1 }
{ M_{ - Q_d ( p ) }^\lambda
    ( \| f_1 \|_1 , \| f_2 \|_1 , \dots , \| f_m \|_1 ) }.
\end{align}
Moreover, the power transform $ Q_d ( p ) $ satisfies
\begin{align}
- Q_d ( p )
= - \frac{ p }{ 1 - d p }
= \frac{ \alpha }{ 1 + d \alpha }.
\end{align}
Here we use the same convention as in Theorem~\ref{thm:multi input BBL}: $ \alpha / ( 1 + d \alpha ) = - \infty $ for $ \alpha = - 1 / d $ and $ \alpha / ( 1 + d \alpha ) = 1 / d $ for $ \alpha = \infty $.
Therefore, the equality \eqref{eq:norm same output} follows.

Next, we derive Theorem~\ref{thm:multi input BBL} from Theorem~\ref{thm:main-BBL}.
If the hypothesis \eqref{eq:multi input BBL assumption} of Theorem~\ref{thm:multi input BBL} holds, then
\begin{align}
    K = \einf_{ x \in S } L ( x ) \geq 1
\end{align}
by \eqref{eq:K transform}.
By applying Theorem~\ref{thm:main-BBL}, we get
\begin{align}
    M_{ Q_d ( p ) }^\lambda \left( \frac{ \| g \|_1 }{ \| f_1 \|_1 } , \frac{ \| g \|_1 }{ \| f_2 \|_1 } , \dots , \frac{ \| g \|_1 }{ \| f_m \|_1 } \right) 
    \geq 1.
\end{align}
Thus, the equality \eqref{eq:norm same output} yields \eqref{eq:multi input BBL claim}, and hence Theorem~\ref{thm:multi input BBL} follows.

Conversely, we derive the common-output case $ g = g_1 = g_2 = \dots = g_m $ of Theorem~\ref{thm:main-BBL} from Theorem~\ref{thm:multi input BBL}.
The assertion is immediate if $ K = 0 $, so we may assume $ K > 0 $.
Then \eqref{eq:K transform} gives
\begin{align}
    M_\alpha^\lambda ( f_1 ( x_1 ) , f_2 ( x_2 ) , \dots , f_m ( x_m ) ) \leq \frac{ g ( z_\lambda ( x ) ) }{ K } & & \text{ a.e. $ x = ( x_1 , \dots , x_m ) \in S $ }.
\end{align}
By applying Theorem~\ref{thm:multi input BBL} with $ g $ replaced by $ g / K $, we obtain
\begin{align}
    M_{ \alpha / ( 1 + d \alpha ) }^\lambda ( \| f_1 \|_1 , \| f_2 \|_1 , \dots , \| f_m \|_1 )
    \leq \frac{ \| g \|_1 }{ K }.
\end{align}
By \eqref{eq:norm same output}, we obtain
\begin{align}
    K \leq \frac{ \| g \|_1 }{ M_{ \alpha / ( 1 + d \alpha ) }^\lambda ( \| f_1 \|_1 , \| f_2 \|_1 , \dots , \| f_m \|_1 ) }
    = M_{ Q_d ( p ) }^\lambda \left( \frac{ \| g \|_1 }{ \| f_1 \|_1 } , \frac{ \| g \|_1 }{ \| f_2 \|_1 } , \dots , \frac{ \| g \|_1 }{ \| f_m \|_1 } \right).
\end{align}
This proves Theorem~\ref{thm:main-BBL} in the common-output case.

Finally, we consider proportional outputs in Theorem~\ref{thm:main-BBL}.
Suppose that there exist $ c_1 , c_2 , \dots , c_m > 0 $ and a function $ g \colon \mathbb{ R }^d \to \mathbb{ R }_{ \geq 0 } $ such that $ g_i = c_i g $ for every $ i = 1 , 2 , \dots , m $.
Then the problem reduces to the common-output case.
Indeed, setting $ \widetilde{ f_i } = f_i / c_i $, we have
\begin{align}
    S ( f_i ) = S ( \widetilde{ f_i } ), & &
\frac{ g_i ( z ) }{ f_i ( x_i ) }
= \frac{ c_i g ( z ) }{ f_i ( x_i ) }
= \frac{ g ( z ) }{ \widetilde{ f_i } ( x_i ) }.
\end{align}
Thus, the value of $ L ( x ) $ is unchanged by this replacement.
Moreover, since
\begin{align}
    \frac{ \| g_i \|_1 }{ \| f_i \|_1 }
    = \frac{ c_i \| g \|_1 }{ \| f_i \|_1 }
    = \frac{ \| g \|_1 }{ \| \widetilde{ f_i } \|_1 },
\end{align}
both sides of \eqref{eq:main-BBL-claim} are unchanged.
Thus, the proportional-output case reduces to the common-output case.

\subsection{Derivation of the main theorem in the low-power range from the classical Borell--Brascamp--Lieb inequality for multiple functions}
\label{subsec:known inequalities low power}

In this subsection, we explain that Theorem~\ref{thm:main-BBL} can be derived from Theorem~\ref{thm:multi input BBL} by combining Lemma~\ref{lem:Holder} with Lemma~\ref{lem:integral inequality} in the low-power range $ - \infty \leq p \leq 1 / ( d + 1 ) $.
The assertion is immediate if $ K = 0 $, so we may assume $ K > 0 $.
For $ i = 1 , 2 , \dots , m $, we put $ h_i \coloneqq g_i / \| f_i \|_1 $ and define $ \Gamma ( z ) \coloneqq M_{ Q_d ( p ) }^\lambda ( h_1 ( z ) , h_2 ( z ) , \dots , h_m ( z ) ) $.
Then by Lemma~\ref{lem:Holder}, we have
\begin{align}
    K
    \leq L ( x )
    \leq M_{ 1 / d }^\lambda \left( \frac{ \| f_1 \|_1 }{ f_1 ( x_1 ) } , \frac{ \| f_2 \|_1 }{ f_2 ( x_2 ) } , \dots , \frac{ \| f_m \|_1 }{ f_m ( x_m ) } \right) \Gamma ( z_\lambda ( x ) )
\end{align}
for almost every $ x = ( x_1 , x_2 , \dots , x_m ) \in S $.
By Fact~\ref{fact:weighted power mean properties} \ref{item:weighted power mean properties inverse}, we get
\begin{align}
    M_{ 1 / d }^\lambda \left( \frac{ \| f_1 \|_1 }{ f_1 ( x_1 ) } , \frac{ \| f_2 \|_1 }{ f_2 ( x_2 ) } , \dots , \frac{ \| f_m \|_1 }{ f_m ( x_m ) } \right)
    = M_{ - 1 / d }^\lambda \left( \frac{ f_1 ( x_1 ) }{ \| f_1 \|_1 } , \frac{ f_2 ( x_2 ) }{ \| f_2 \|_1 } , \dots , \frac{ f_m ( x_m ) }{ \| f_m \|_1 } \right)^{ - 1 },
\end{align}
and therefore
\begin{align}
    M_{ - 1 / d }^\lambda \left( \frac{ f_1 ( x_1 ) }{ \| f_1 \|_1 } , \frac{ f_2 ( x_2 ) }{ \| f_2 \|_1 } , \dots , \frac{ f_m ( x_m ) }{ \| f_m \|_1 } \right)
    \leq \frac{ \Gamma ( z_\lambda ( x ) ) }{ K }.
\end{align}
If $ \| \Gamma \|_1 < \infty $, then Theorem~\ref{thm:multi input BBL} yields
\begin{align}
    \frac{ \| \Gamma \|_1 }{ K }
    \geq M_{ - \infty }^\lambda \left( \frac{ \| f_1 \|_1 }{ \| f_1 \|_1 } , \frac{ \| f_2 \|_1 }{ \| f_2 \|_1 } , \dots , \frac{ \| f_m \|_1 }{ \| f_m \|_1 } \right)
    = M_{ - \infty }^\lambda ( 1 , 1 , \dots , 1 )
    = 1,
\end{align}
so $ K \leq \| \Gamma \|_1 $.
This inequality $ K \leq \| \Gamma \|_1 $ is trivially true when $ \| \Gamma \|_1 = \infty $.

Applying Lemma~\ref{lem:applicability} with $ e = 1 $, we obtain
    \begin{align}
        Q_d ( p ) \leq 1
        \Longleftrightarrow
        p \leq \frac{ 1 }{ d + 1 }. \label{eq:applicable range}
    \end{align}
Thus, it follows from \eqref{eq:applicable range} and $ p \leq 1 / ( d + 1 ) $ that $ Q_d ( p ) \leq 1 $.
By applying Lemma~\ref{lem:integral inequality} with its power $ p $ replaced by $ Q_d ( p ) $, we see that $ \Gamma $ is integrable and
\begin{align}
    \| \Gamma \|_1
    \leq M_{ Q_d ( p ) }^\lambda ( \| h_1 \|_1 , \| h_2 \|_1 , \dots , \| h_m \|_1 ). \label{eq:crucial}
\end{align}
Thus, it follows from $ h_i = g_i / \| f_i \|_1 $ that
\begin{align}
    \| \Gamma \|_1
    \leq M_{ Q_d ( p ) }^\lambda \left( \frac{ \| g_1 \|_1 }{ \| f_1 \|_1 } , \frac{ \| g_2 \|_1 }{ \| f_2 \|_1 } , \dots , \frac{ \| g_m \|_1 }{ \| f_m \|_1 } \right).
\end{align}
Combining this with $ K \leq \| \Gamma \|_1 $ proves Theorem~\ref{thm:main-BBL}.

\begin{remark}
    \label{rem:high power}

    Even in the high-power range $ 1 / ( d + 1 ) < p \leq 1 / d $, the preceding argument still yields $ K \leq \| \Gamma \|_1 $.
    On the other hand, by \eqref{eq:applicable range}, one has $ Q_d ( p ) > 1 $ when $ 1 / ( d + 1 ) < p \leq 1 / d $.
    Thus, Lemma~\ref{lem:integral inequality} cannot be applied with its power $ p $ replaced by $ Q_d ( p ) $ (Remark~\ref{rem:integral inequality sharpest}).
    Consequently, the estimate \eqref{eq:crucial} is unavailable, and the above method combining Lemmas~\ref{lem:Holder} and \ref{lem:integral inequality} does not yield Theorem~\ref{thm:main-BBL}.

    Nevertheless, Theorem~\ref{thm:main-BBL} itself remains valid in the high-power range $ 1 / ( d + 1 ) < p \leq 1 / d $.
    In Sections~\ref{sec:one dimension} and \ref{sec:high dimension}, we prove Theorem~\ref{thm:main-BBL} independently of the preceding derivation.
This proof covers the full range $ - \infty \leq p \leq 1 / d $.
    
\end{remark}

\subsection{Multi-output Pr\'ekopa--Leindler inequality of Cordero-Erausquin--Maurey}
\label{subsec:known inequalities Cordero-Erausquin--Maurey}

Cordero-Erausquin--Maurey established a Pr\'ekopa--Leindler inequality for several input and output functions connected by a general contractive linear operator.
We describe this generalized Pr\'ekopa--Leindler inequality (Fact~\ref{fact:Cordero-Erausquin--Maurey}).
We also explain that Fact~\ref{fact:Cordero-Erausquin--Maurey} contains the case $ p = 0 $ of Theorem~\ref{thm:main-BBL} with the essential infimum in \eqref{eq:main theorem essential inf} replaced by the ordinary infimum over $ S $.

Cordero-Erausquin--Maurey proved the following inequality.

\begin{fact}[Cordero-Erausquin--Maurey {\cite[Theorem 1.6]{MR3652434}}]
\label{fact:Cordero-Erausquin--Maurey}

Suppose that 
\begin{align}
    \lambda^\mathrm{ I } = ( \lambda_1^\mathrm{ I } , \lambda_2^\mathrm{ I } , \dots , \lambda_{ m_\mathrm{ I } }^\mathrm{ I } ), & &
     \lambda^\mathrm{ O } = ( \lambda_1^\mathrm{ O } , \lambda_2^\mathrm{ O } , \dots , \lambda_{ m_\mathrm{ O } }^\mathrm{ O } )
\end{align}
are weights, and $ A \colon ( \mathbb{ R }^d )^{ m_\mathrm{ I } } \to ( \mathbb{ R }^d )^{ m_\mathrm{ O } } $ is a linear operator satisfying
\begin{align}
    A ( w , w , \dots , w ) = ( w , w , \dots , w ) \label{eq:Cordero-Erausquin--Maurey same point}
\end{align}
for any $ w \in \mathbb{ R }^d $.
Let $ \Phi_i^\mathrm{ I } \colon \mathbb{ R }^d \to ( - \infty , \infty ] $ be a Borel measurable function such that $ \exp ( - \Phi_i^\mathrm{ I } ( x ) ) $ is integrable for every $ i = 1 , 2 , \dots , m_\mathrm{ I } $.
Similarly, let $ \Phi_j^\mathrm{ O } \colon \mathbb{ R }^d \to ( - \infty , \infty ] $ be a Borel measurable function such that $ \exp ( - \Phi_j^\mathrm{ O } ( x ) ) $ is integrable for every $ j = 1 , 2 , \dots , m_\mathrm{ O } $.
Assume that
\begin{align}
    M_2^{ \lambda^\mathrm{ O } } ( | ( A x )_1 | , | ( A x )_2 | , \dots , | ( A x )_{ m_\mathrm{ O } } | )
    & \leq M_2^{ \lambda^\mathrm{ I } } ( | x_1 | , | x_2 | , \dots , | x_{ m_\mathrm{ I } } | ), \label{eq:Cordero-Erausquin--Maurey convergence condition} \\
    \sum_{ j = 1 }^{ m_\mathrm{ O } } \lambda_j^\mathrm{ O } \Phi_j^\mathrm{ O } ( ( A x )_j )
    & \leq \sum_{ i = 1 }^{ m_\mathrm{ I } } \lambda_i^\mathrm{ I } \Phi_i^\mathrm{ I } ( x_i )
\end{align}
hold for any $ x = ( x_1 , x_2 , \dots , x_{ m_\mathrm{ I } } ) \in ( \mathbb{ R }^d )^{ m_\mathrm{ I } } $.
Here $ | w | $ denotes the Euclidean norm of $ w \in\mathbb{ R }^d $.
Then we have
\begin{align}
    & \phantom{ { } = { } } M_0^{ \lambda^\mathrm{ I } } \left( \int_{ \mathbb{ R }^d } \exp ( - \Phi_1^\mathrm{ I } ( x ) ) \; dx , \dots , \int_{ \mathbb{ R }^d } \exp ( - \Phi_{ m_\mathrm{ I } }^\mathrm{ I } ( x ) ) \; dx \right) \\
    & \leq M_0^{ \lambda^\mathrm{ O } } \left( \int_{ \mathbb{ R }^d } \exp ( - \Phi_1^\mathrm{ O } ( x ) ) \; dx , \dots , \int_{ \mathbb{ R }^d } \exp ( - \Phi_{ m_\mathrm{ O } }^\mathrm{ O } ( x ) ) \; dx \right). \label{eq:Cordero-Erausquin--Maurey claim}
\end{align}

\end{fact}

When $ m_\mathrm{ O } = 1 $, Fact~\ref{fact:Cordero-Erausquin--Maurey} agrees with the Pr\'ekopa--Leindler inequality for multiple functions.
Thus, Fact~\ref{fact:Cordero-Erausquin--Maurey} gives a Pr\'ekopa--Leindler inequality for finitely many input and output functions connected by a general linear operator $ A $.

\begin{remark}

    Cordero-Erausquin--Maurey also discuss further generalizations of Fact~\ref{fact:Cordero-Erausquin--Maurey}.
    The input functions $ \Phi_i^\mathrm{ I } $ and output functions $ \Phi_j^\mathrm{ O } $ need not form finite families.
    That is, the result extends to families indexed by a general finite measure space \cite[Theorem 1.3]{MR3652434}.
    They also treat the case in which the dimensions of the input and output functions differ \cite[Theorem 5.1]{MR3652434}.
    This generalization contains the geometric Brascamp--Lieb inequality \cite{MR412366} and Barthe's reverse Brascamp--Lieb inequality \cite{MR1650312}.
    A framework involving multiple input functions and multiple output functions has also been studied in the form of forward--reverse Brascamp--Lieb inequalities \cite{MR3879918} \cite{MR4236529} \cite{MR4930591}.
    
\end{remark}

Fact~\ref{fact:Cordero-Erausquin--Maurey} implies the case of Theorem~\ref{thm:main-BBL} with $ p = 0 $ when the essential infimum in \eqref{eq:main theorem essential inf} is replaced by the infimum over all $ x \in S $.
Indeed, let $ f_1 , f_2 , \dots , f_m $ and $ g_1 , g_2 , \dots , g_m $ be Borel measurable, and let $ S $ be as in Theorem~\ref{thm:main-BBL}.
Set
\begin{align}
K' \coloneqq \inf_{ x \in S } L ( x ). \label{eq:K' definition}
\end{align}
Then $ K' \leq K $, although equality need not hold in general.
We can derive from Fact~\ref{fact:Cordero-Erausquin--Maurey} that
\begin{align}
    K' \leq M_0^\lambda \left( \frac{ \| g_1 \|_1 }{ \| f_1 \|_1 } , \frac{ \| g_2 \|_1 }{ \| f_2 \|_1 } , \dots , \frac{ \| g_m \|_1 }{ \| f_m \|_1 } \right) \label{eq:main theorem restrict}
\end{align}
as follows.
The assertion is immediate if $ K' = 0 $, so we may assume $ K' > 0 $.
We set
\begin{align}
    m_\mathrm{ I } = m_\mathrm{ O } = m , & &
    \lambda^{ \mathrm{ I } } = \lambda^{ \mathrm{ O } } = \lambda.
\end{align}
We define a linear operator $ A \colon ( \mathbb{ R }^d )^m \to ( \mathbb{ R }^d )^m $ by $ A ( x ) \coloneqq ( z_\lambda ( x ) , z_\lambda ( x ) , \dots , z_\lambda ( x ) ) $ for $ x \in ( \mathbb{ R }^d )^m $.
Then \eqref{eq:Cordero-Erausquin--Maurey same point} holds for any $ w \in \mathbb{ R }^d $.
Moreover, for any $ x = ( x_1 , x_2 , \dots , x_m ) \in ( \mathbb{ R }^d )^m $,
\begin{align}
    M_2^{ \lambda^\mathrm{ O } } ( | ( A x )_1 | , \dots , | ( A x )_{ m_\mathrm{ O } } | )
    = M_2^\lambda ( | z_\lambda ( x ) | , \dots , | z_\lambda ( x ) | )
    = | z_\lambda ( x ) |. \label{eq:convergence output}
\end{align}
By the triangle inequality for the Euclidean norm,
\begin{align}
    | z_\lambda ( x ) |
    = \left| \sum_{ i = 1 }^m \lambda_i x_i \right|
    \leq \sum_{ i = 1 }^m \lambda_i | x_i |
    = M_1^\lambda ( | x_1 | , | x_2 | , \dots , | x_m | ).
\end{align}
Thus, Fact~\ref{fact:weighted power mean properties} \ref{item:weighted power mean properties power monotone main} gives
\begin{align}
    | z_\lambda ( x ) |
    \leq M_2^{ \lambda^\mathrm{ I } } ( | x_1 | , | x_2 | , \dots , | x_{ m_\mathrm{ I } } | ). \label{eq:convergence input}
\end{align}
Combining \eqref{eq:convergence output} and \eqref{eq:convergence input}, we obtain \eqref{eq:Cordero-Erausquin--Maurey convergence condition}.
For $ i = 1 , 2 , \dots , m $, we define
\begin{align}
    \Phi_i^\mathrm{ I } \coloneqq - \log f_i , & &
    \Phi_i^\mathrm{ O } \coloneqq - \log \frac{ g_i }{ K' } ,
\end{align}
with the convention $ - \log 0 = \infty $.
If $ x \in S $, then \eqref{eq:K' definition} gives $ K' \leq L ( x ) $.
Therefore
\begin{align}
\sum_{ j = 1 }^{ m_\mathrm{ O } } \lambda_j^\mathrm{ O } \Phi_j^\mathrm{ O } ( ( A x )_j )
    = \sum_{ j = 1 }^m \lambda_j \Phi_j^\mathrm{ O } ( z_\lambda ( x ) )
    \leq \sum_{ i = 1 }^m \lambda_i \Phi_i^\mathrm{ I } ( x_i )
    = \sum_{ i = 1 }^{ m_\mathrm{ I } } \lambda_i^\mathrm{ I } \Phi_i^\mathrm{ I } ( x_i ) . \label{eq:K' log version}
\end{align}
If $ x \notin S $, then
\begin{align}
  \sum_{ i = 1 }^m \lambda_i \Phi_i^\mathrm{ I } ( x_i ) = \infty ,
\end{align}
so \eqref{eq:K' log version} holds for any $ x = ( x_1 , x_2 , \dots , x_m ) \in ( \mathbb{ R }^d )^m $.
Thus, Fact~\ref{fact:Cordero-Erausquin--Maurey} gives \eqref{eq:Cordero-Erausquin--Maurey claim}.
For any $ i = 1 , 2 , \dots , m $, we have
\begin{align}
     \exp ( - \Phi_i^\mathrm{ I } ( x ) )
     = f_i ( x ) , & &
     \exp ( - \Phi_i^\mathrm{ O } ( x ) )
     = \frac{ g_i ( x ) }{ K' } .
\end{align}
Thus, \eqref{eq:Cordero-Erausquin--Maurey claim} implies
\begin{align}
    \prod_{ i = 1 }^m \| f_i \|_1^{ \lambda_i }
    \leq \prod_{ i = 1 }^m \left( \frac{ \| g_i \|_1 }{ K' } \right)^{ \lambda_i }
    = \frac{ 1 }{ K' } \prod_{ i = 1 }^m \| g_i \|_1^{ \lambda_i },
\end{align}
which proves \eqref{eq:main theorem restrict}.

The quantity $ K $ in \eqref{eq:main theorem essential inf} is instead defined using an essential infimum.
Therefore, after replacing $ K' $ by $ K $ in \eqref{eq:K' log version}, the corresponding inequality generally holds only almost everywhere on the product space $ S $.
Since Cordero-Erausquin--Maurey assume \eqref{eq:K' log version} for any $ x \in ( \mathbb{ R }^d )^m $, the case $ p = 0 $ of Theorem~\ref{thm:main-BBL} does not follow directly from Fact~\ref{fact:Cordero-Erausquin--Maurey}.

The inequality \eqref{eq:main theorem restrict} is the case $ p = 0 $ of Theorem~\ref{thm:main-BBL} with the essential infimum replaced by the ordinary infimum over $ S $.
It follows from Fact~\ref{fact:Cordero-Erausquin--Maurey} by choosing the operator so that all output components are evaluated at the common point $ z_\lambda ( x ) $.
Fact~\ref{fact:Cordero-Erausquin--Maurey} allows distinct output functions to be evaluated at different linear combinations, but its integral conclusion is restricted to $ p = 0 $.
Cordero-Erausquin--Maurey prove Fact~\ref{fact:Cordero-Erausquin--Maurey} by Borell's probabilistic method.

By contrast, Theorem~\ref{thm:main-BBL} restricts the output points to be common but gives an integral inequality involving the weighted power mean $ M_p^\lambda $ for the full range $ - \infty \leq p \leq 1 / d $.
Our proof of Theorem~\ref{thm:main-BBL} establishes the one-dimensional case through a minimum-type integral inequality for multiple functions (Lemma~\ref{lem:multifunction-minimum-integral}) and the higher-dimensional case through a dimension-additivity lemma based on slice integrals (Lemma~\ref{lem:dimension-additivity}).
Table~\ref{tab:Cordero-Erausquin--Maurey comparison} compares Fact~\ref{fact:Cordero-Erausquin--Maurey} with Theorem~\ref{thm:main-BBL}.

\begin{table}[htbp]
    \centering
    \caption{Comparison of Fact~\ref{fact:Cordero-Erausquin--Maurey} and Theorem~\ref{thm:main-BBL}}
    \label{tab:Cordero-Erausquin--Maurey comparison}
    \bigskip
    \begin{tabular}{
        |p{0.25\linewidth}
        ||p{0.33\linewidth}
        |p{0.33\linewidth}|
    }
        \hline
        & Fact~\ref{fact:Cordero-Erausquin--Maurey} & Theorem~\ref{thm:main-BBL} \\
        \hline \hline

        Number of functions &
        Finite (extendable to families indexed by a general finite measure space) &
        Finite
        \\
        \hline

        Numbers of input and output functions & 
        May be different &
        Equal
        \\
        \hline

        Range of powers &
        $ p = 0 $ &
        $ - \infty \leq p \leq 1 / d $
        \\
        \hline

        Points where the hypothesis is imposed &
        All input points &
        Almost everywhere
        \\
        \hline

        Dimensions of the domains of input and output functions &
        Same (extendable to different dimensions) &
        Same
        \\
        \hline

        Method of proof &
        Borell's probabilistic method &
        Minimum-type integral inequality and slice integrals
        \\
        \hline
    \end{tabular}
\end{table}

\section{Proof of the main theorem in dimension one}
\label{sec:one dimension}

In this section, we prove the one-dimensional case of Theorem~\ref{thm:main-BBL}.
In Section~\ref{subsec:one dimension two function}, we explain the one-dimensional two-function minimum-type integral inequality of Brascamp--Lieb (Fact~\ref{fact:Brascamp--Lieb}) and prove preservation of the $ L^\infty $ norm (Lemma~\ref{lem:B esssup}).
In Section~\ref{subsec:one dimension multifunction}, by using the results of Section~\ref{subsec:one dimension two function}, we prove a minimum-type integral inequality for multiple functions (Lemma~\ref{lem:multifunction-minimum-integral}) needed for the proof of Theorem~\ref{thm:main-BBL}.
In Section~\ref{subsec:one dimension proof}, by combining this integral inequality with Lemmas~\ref{lem:Holder}, \ref{lem:Holder infinity}, and \ref{lem:integral inequality}, we complete the proof of Theorem~\ref{thm:main-BBL} in dimension one.

\subsection{Two-function minimum-type integral inequality of Brascamp--Lieb}
\label{subsec:one dimension two function}

In this subsection, we introduce the two-function minimum-type integral inequality of Brascamp--Lieb.

\begin{fact}[Brascamp--Lieb {\cite[Theorem 3.1 and Theorem A.2]{MR450480}}]
    \label{fact:Brascamp--Lieb}

Let $ f_1 , f_2 \colon \mathbb{ R } \to \mathbb{ R }_{ \geq 0 } $ be measurable functions satisfying
\begin{align}
    T \coloneqq \| f_1 \|_\infty = \| f_2 \|_\infty, \label{eq:Brascamp--Lieb T definition}
\end{align}
and let $ \lambda = ( \lambda_1 , \lambda_2 ) $ be a weight.
Then the function $ B \colon \mathbb{ R } \to [ 0 , \infty ] $ defined by
\begin{align}
    B ( y ) \coloneqq \esup_{ x_2 \in \mathbb{ R } } \min \left\{ f_1 \left( \frac{ y - \lambda_2 x_2 }{ \lambda_1 } \right) , f_2 ( x_2 ) \right\}
\end{align}
is lower semicontinuous and hence measurable, and satisfies
\begin{align}
    \lambda_1 \| f_1 \|_1 + \lambda_2 \| f_2 \|_1 \leq \| B \|_1.
\end{align}

\end{fact}

Dancs--Uhrin \cite[Theorem 2.1]{MR572660} prove a more general one-dimensional integral inequality than Fact~\ref{fact:Brascamp--Lieb} in a formulation with the ordinary supremum.
They also state without proof that their result can be extended to an arbitrary finite number of functions.
They further state that the same proof applies when the ordinary supremum is replaced by the essential supremum.
Equality cases for Fact~\ref{fact:Brascamp--Lieb} were given by Dubuc \cite[Th\'eor\`eme 9]{MR444863} and Rossi--Salani \cite{MR3977216}.
Rossi--Salani also studied stability near equality.

We also prove the following lemma, which will be needed in the proof of the minimum-type integral inequality for multiple functions (Lemma~\ref{lem:multifunction-minimum-integral}).

\begin{lemma}
    \label{lem:B esssup}

    Let $ f_1 , f_2 , T , \lambda $, and $ B $ be as in Fact~\ref{fact:Brascamp--Lieb}.
    Then
    \begin{align}
        \| B \|_\infty = T. \label{eq:B esssup claim}
    \end{align}
    
\end{lemma}

\begin{proof}

For $ \tau \geq 0 $, the inequality $ \| B \|_\infty > \tau $ is equivalent to
\begin{align}
\int_{ \mathbb{ R } } \int_{ \mathbb{ R } } \chi_{ \{ f_1 > \tau \} } \left( \frac{ y - \lambda_2 x_2 }{ \lambda_1 } \right) \chi_{ \{ f_2 > \tau \} } ( x_2 ) \; dx_2 \; dy
> 0.
\end{align}
By Tonelli's theorem, this is in turn equivalent to
\begin{align}
    | \{ f_1 > \tau \} | > 0 , & &
    | \{ f_2 > \tau \} | > 0. \label{eq:B esssup proof level sets}
\end{align}
By \eqref{eq:Brascamp--Lieb T definition}, the condition \eqref{eq:B esssup proof level sets} is equivalent to $ T > \tau $.
Thus, \eqref{eq:B esssup claim} follows.
\end{proof}

\subsection{Minimum-type integral inequality for multiple functions}
\label{subsec:one dimension multifunction}

In this subsection, we prove the minimum-type integral inequality for multiple functions (Lemma~\ref{lem:multifunction-minimum-integral}) needed for the proof of Theorem~\ref{thm:main-BBL} in dimension one by using Fact~\ref{fact:Brascamp--Lieb}.

Dancs--Uhrin \cite{MR572660} state without proof that Fact~\ref{fact:Brascamp--Lieb} can be extended to an arbitrary finite number of functions by the same method.
We use a version of the corresponding multiple-function result for our application to Theorem~\ref{thm:main-BBL}.
Rather than generalizing the method of Brascamp--Lieb or Dancs--Uhrin directly, we prove it by induction on the number of functions by using Fact~\ref{fact:Brascamp--Lieb} as follows.

\begin{lemma}[Minimum-type integral inequality for multiple functions]
\label{lem:multifunction-minimum-integral}

Let 
\begin{align}
    f_1 , f_2 , \dots , f_m \colon \mathbb{ R } \to \mathbb{ R }_{ \geq 0 } 
\end{align}
be measurable functions satisfying
\begin{align}
    T \coloneqq \| f_1 \|_\infty = \| f_2 \|_\infty = \dots = \| f_m \|_\infty. \label{eq:multifunction-minimum-integral assume}
\end{align}
For any weight $ \lambda = ( \lambda_1 , \lambda_2 , \dots , \lambda_m ) $ and measurable function $ H \colon \mathbb{ R } \to \mathbb{ R }_{ \geq 0 } $ such that
\begin{align}
    M_{ - \infty }^\lambda ( f_1 ( x_1 ) , f_2 ( x_2 ) , \dots , f_m ( x_m ) )
    \leq H ( z_\lambda ( x ) ) \label{eq:multifunction-minimum-integral H large}
\end{align}
for almost every $ x = ( x_1 , x_2 , \dots , x_m ) \in \mathbb{ R }^m $, one has
\begin{align}
    \sum_{ i = 1 }^m \lambda_i \| f_i \|_1 \leq \| H \|_1. \label{eq:multifunction-minimum-integral claim}
\end{align}

\end{lemma}

\begin{proof}

The case of $ T = \infty $ follows from the case of $ T < \infty $.
Indeed, we consider $ \min \{ f_i , \tau \} $ in place of each $ f_i $ for $ \tau > 0 $.
The monotone convergence theorem then yields the desired conclusion as $ \tau \to \infty $ by using the case of $ T < \infty $.
Thus, it suffices to assume $ T < \infty $.

We proceed by induction on $ m $.
For $ m = 1 $, integrating both sides of \eqref{eq:multifunction-minimum-integral H large} gives \eqref{eq:multifunction-minimum-integral claim}.

We assume that the assertion holds for any weight with $ m $ components and prove the case of $ m + 1 $.
For this case, we take a weight $ \lambda = ( \lambda_1 , \lambda_2 , \dots , \lambda_{ m + 1 } ) $ and
\begin{align}
T \coloneqq \| f_1 \|_\infty = \| f_2 \|_\infty = \dots = \| f_{ m + 1 } \|_\infty < \infty. \label{eq:multifunction-minimum-integral proof essentially sup}
\end{align}
We further assume that
\begin{align}
M_{ - \infty }^\lambda ( f_1 ( x_1 ) , f_2 ( x_2 ) , \dots , f_{ m + 1 } ( x_{ m + 1 } ) )
\leq H ( z_\lambda ( x ) ) \label{eq:multifunction-minimum-integral proof assume}
\end{align}
for almost every $ x = ( x_1 , x_2 , \dots , x_{ m + 1 } ) \in \mathbb{ R }^{ m + 1 } $.
It suffices to prove
\begin{align}
\sum_{ i = 1 }^{ m + 1 } \lambda_i \| f_i \|_1
\leq \| H \|_1. \label{eq:multifunction-minimum-integral proof conclude}
\end{align}
We use the notation
\begin{align}
\Lambda \coloneqq \lambda_1 + \lambda_2, & &
\widehat{ \lambda }
= ( \widehat{ \lambda_1 } , \widehat{ \lambda_2 } )
\coloneqq \left( \frac{ \lambda_1 }{ \Lambda } , \frac{ \lambda_2 }{ \Lambda } \right), & &
\widetilde{ \lambda } \coloneqq ( \Lambda , \lambda_3 , \lambda_4 , \dots , \lambda_{ m + 1 } ).
\end{align}
Then $ \widehat{ \lambda } $ and $ \widetilde{ \lambda } $ are weights.
By Fact~\ref{fact:Brascamp--Lieb}, the function
\begin{align}
    B ( y ) \coloneqq \esup_{ x_2 \in \mathbb{ R } } \min \left\{ f_1 \left( \frac{ y - \widehat{ \lambda_2 } x_2 }{ \widehat{ \lambda_1 } } \right) , f_2 ( x_2 ) \right\}
\end{align}
is measurable and satisfies
\begin{align}
    \| B \|_1
    \geq \widehat{ \lambda_1 } \| f_1 \|_1 + \widehat{ \lambda_2 } \| f_2 \|_1
    = \frac{ \lambda_1 \| f_1 \|_1 + \lambda_2 \| f_2 \|_1 }{ \Lambda }.
\end{align}
Thus, we have
\begin{align}
    \sum_{ i = 1 }^{ m + 1 } \lambda_i \| f_i \|_1
    = \lambda_1 \| f_1 \|_1 + \lambda_2 \| f_2 \|_1 + \sum_{ i = 3 }^{ m + 1 } \lambda_i \| f_i \|_1
    \leq \Lambda \| B \|_1 + \sum_{ i = 3 }^{ m + 1 } \lambda_i \| f_i \|_1. \label{eq:multifunction-minimum-integral proof use B}
\end{align}
Moreover, since Lemma~\ref{lem:B esssup} and \eqref{eq:multifunction-minimum-integral proof essentially sup} imply $ \| B \|_\infty = T $, we have
\begin{align}
    B ( y ) \leq \esup_{ x_2 \in \mathbb{ R } } f_2 ( x_2 ) = T < \infty
\end{align}
for any $ y \in \mathbb{ R } $.

Next, we prove
\begin{align}
    M_{ - \infty }^{ \widetilde{ \lambda } } ( B ( y ) , f_3 ( x_3 ) , f_4 ( x_4 ) , \dots , f_{ m + 1 } ( x_{ m + 1 } ) )
    \leq H ( z_{ \widetilde{ \lambda } } ( x' ) ) \label{eq:multifunction-minimum-integral proof compose}
\end{align}
for almost every $ x' = ( y , x_3 , x_4 , \dots , x_{ m + 1 } ) \in \mathbb{ R }^m $.
If $ y = \widehat{ \lambda_1 } x_1 + \widehat{ \lambda_2 } x_2 $, then $ x_1 = ( y - \widehat{ \lambda_2 } x_2 ) / \widehat{ \lambda_1 } $.
Thus, the linear transformation
\begin{align}
    ( x_1 , x_2 , x_3 , \dots , x_{ m + 1 } )
    \mapsto ( y , x_2 , x_3 , \dots , x_{ m + 1 } )
\end{align}
is invertible.
The identity $ \lambda_1 x_1 + \lambda_2 x_2 = \Lambda y $ holds.
Thus, it follows from \eqref{eq:multifunction-minimum-integral proof assume} that
\begin{align}
H ( z_{ \widetilde{ \lambda } } ( x' ) )
& = H \left( \Lambda y + \sum_{ i = 3 }^{ m + 1 } \lambda_i x_i \right) \\
& \geq M_{ - \infty }^\lambda \left( f_1 \left( \frac{ y - \widehat{ \lambda_2 } x_2 }{ \widehat{ \lambda_1 } } \right) , f_2 ( x_2 ) , f_3 ( x_3 ) , \dots , f_{ m + 1 } ( x_{ m + 1 } ) \right) \\
& = M_{ - \infty }^{ \widetilde{ \lambda } } \left( \min \left\{ f_1 \left( \frac{ y - \widehat{ \lambda_2 } x_2 }{ \widehat{ \lambda_1 } } \right) , f_2 ( x_2 ) \right\} , f_3 ( x_3 ) , \dots , f_{ m + 1 } ( x_{ m + 1 } ) \right)
\end{align}
for almost every $ ( x' , x_2 ) \in \mathbb{ R }^m \times \mathbb{ R } $.
Thus, Fubini's theorem and the definition of $ B $ imply that \eqref{eq:multifunction-minimum-integral proof compose} holds for almost every $ x' \in \mathbb{ R }^m $.
The induction hypothesis for $ m $ functions gives
\begin{align}
\Lambda \| B \|_1 + \sum_{ i = 3 }^{ m + 1 } \lambda_i \| f_i \|_1
\leq \| H \|_1.
\end{align}
By combining this with \eqref{eq:multifunction-minimum-integral proof use B}, we obtain \eqref{eq:multifunction-minimum-integral proof conclude}.
\end{proof}

\subsection{Completion of the proof of the main theorem in dimension one}
\label{subsec:one dimension proof}

In this subsection, we complete the proof of Theorem~\ref{thm:main-BBL} in dimension one by using Lemmas~\ref{lem:Holder}, \ref{lem:Holder infinity}, \ref{lem:integral inequality}, and \ref{lem:multifunction-minimum-integral}.
The assertion is immediate if $ K = 0 $, so we may assume $ K > 0 $.

It suffices to prove the assertion under the additional assumption that all the functions $ f_1 , f_2 , \dots , f_m $ are essentially bounded.
This reduction follows from the monotone convergence theorem and Lemma~\ref{lem:weighted power mean continuous}.
If necessary, we may assume without loss of generality that \eqref{eq:multifunction-minimum-integral assume} holds by replacing $ f_i $ and $ g_i $ by $ f_i / \| f_i \|_\infty $ and $ g_i / \| f_i \|_\infty $ for $ i = 1 , 2 , \dots , m $.

For
\begin{align}
    x = ( x_1 , x_2 , \dots , x_m ) \in S , & & z = z_\lambda ( x ),
\end{align}
Lemma~\ref{lem:Holder infinity} gives
\begin{align}
    L ( x ) M_{ - \infty }^\lambda ( f_1 ( x_1 ) , \dots , f_m ( x_m ) )
    \leq M_p^\lambda ( g_1 ( z ) , \dots , g_m ( z ) ).
\end{align}
We define $ H ( y ) \coloneqq M_p^\lambda ( g_1 ( y ) , \dots , g_m ( y ) ) / K $.
Then
\begin{align}
    M_{ - \infty }^\lambda ( f_1 ( x_1 ) , \dots , f_m ( x_m ) )
    \leq \frac{ M_p^\lambda ( g_1 ( z_\lambda ( x ) ) , \dots , g_m ( z_\lambda ( x ) ) ) }{ L ( x ) }
    \leq H ( z_\lambda ( x ) )
\end{align}
for almost every
\begin{align}
    x = ( x_1 , x_2 , \dots , x_m ) \in S.
\end{align}
If $ ( x_1 , x_2 , \dots , x_m ) \notin S $, then $ M_{ - \infty }^\lambda ( f_1 ( x_1 ) , f_2 ( x_2 ) , \dots , f_m ( x_m ) ) = 0 $.
Thus, the inequality \eqref{eq:multifunction-minimum-integral H large} holds for almost every $ x = ( x_1 , x_2 , \dots , x_m ) \in \mathbb{ R }^m $.
By Lemma~\ref{lem:multifunction-minimum-integral}, we have
\begin{align}
    M_1^\lambda ( \| f_1 \|_1 , \dots , \| f_m \|_1 )
    \leq \| H \|_1
    = \frac{ 1 }{ K } \int_{ \mathbb{ R } } M_p^\lambda ( g_1 ( y ) , \dots , g_m ( y ) ) \; dy.
\end{align}
Since $ - \infty \leq p \leq 1 $, Lemma~\ref{lem:integral inequality} gives
\begin{align}
    \int_{ \mathbb{ R } } M_p^\lambda ( g_1 ( y ) , \dots , g_m ( y ) ) \; dy
    \leq M_p^\lambda ( \| g_1 \|_1 , \dots , \| g_m \|_1 ),
\end{align}
and therefore
\begin{align}
K M_1^\lambda ( \| f_1 \|_1 , \dots , \| f_m \|_1 )
    \leq M_p^\lambda ( \| g_1 \|_1 , \dots , \| g_m \|_1 ). \label{eq:one dimension proof f g compare}
\end{align}
On the other hand, Lemma~\ref{lem:Holder} gives
\begin{align}
    M_p^\lambda ( \| g_1 \|_1 , \dots , \| g_m \|_1 )
    \leq M_1^\lambda ( \| f_1 \|_1 , \dots , \| f_m \|_1 ) M_{ Q_1 ( p ) }^\lambda \left( \frac{ \| g_1 \|_1 }{ \| f_1 \|_1 } , \dots , \frac{ \| g_m \|_1 }{ \| f_m \|_1 } \right). \label{eq:one dimension proof Holder}
\end{align}
Since $ \| f_i \|_1 > 0 $ for every $ i = 1 , 2 , \dots , m $, we have
\begin{align}
    M_1^\lambda ( \| f_1 \|_1 , \dots , \| f_m \|_1 ) > 0.
\end{align}
By \eqref{eq:one dimension proof f g compare} and \eqref{eq:one dimension proof Holder}, we obtain \eqref{eq:main-BBL-claim}.
Thus, Theorem~\ref{thm:main-BBL} is proved in dimension one.

\section{Proof of the main theorem in higher dimensions using slice integrals}
\label{sec:high dimension}

In this section, we prove the general case of Theorem~\ref{thm:main-BBL} using its one-dimensional case.
The proof proceeds by induction on the dimension.
We decompose Euclidean space as $ \mathbb{ R }^{ d + e } = \mathbb{ R }^d \times \mathbb{ R }^e $ and apply the results in dimensions $ d $ and $ e $ successively to the slice integrals
\begin{align}
    F_i ( \eta ) \coloneqq \int_{ \mathbb{ R }^d } f_i ( \xi , \eta ) \; d \xi .
\end{align}

For fixed $ \lambda $, $ d $, and $ p $, we denote by $ \mathsf{ P }_\lambda ( d , p ) $ the assertion of Theorem~\ref{thm:main-BBL} restricted to these parameters.
More precisely, for a weight $ \lambda = ( \lambda_1 , \lambda_2 , \dots , \lambda_m ) $, $ d = 1 , 2 , \dots $, and $ - \infty \leq p \leq 1 / d $, the statement $ \mathsf{ P }_\lambda ( d , p ) $ is the following.

\noindent
Statement $ \mathsf{ P }_\lambda ( d , p ) $:
Let $ f_1 , f_2 , \dots , f_m $ and $ S $ be as in Theorem~\ref{thm:multi input BBL}.
For any integrable functions $ g_1 , g_2 , \dots , g_m \colon \mathbb{ R }^d \to \mathbb{ R }_{ \geq 0 } $, we define $ L $ and $ K $ by \eqref{eq:main theorem essential inf}.
Then
\begin{align}
K
\leq M_{ Q_d ( p ) }^\lambda \left( \frac{ \| g_1 \|_1 }{ \| f_1 \|_1 } , \dots , \frac{ \| g_m \|_1 }{ \| f_m \|_1 } \right).
\end{align}

Then we have the following lemma.

\begin{lemma}[Dimension-additivity lemma]
    \label{lem:dimension-additivity}

    For any weight $ \lambda $ and
    \begin{align}
        d , e = 1 , 2 , \dots , & &
        - \infty \leq p \leq 1 / ( d + e ),
    \end{align}
    the statements $ \mathsf{ P }_\lambda ( d , p ) $ and $ \mathsf{ P }_\lambda ( e , Q_d ( p ) ) $ are well-defined, that is, $ p \leq 1 / d $ and $ Q_d ( p ) \leq 1 / e $.
    If both $ \mathsf{ P }_\lambda ( d , p ) $ and $ \mathsf{ P }_\lambda ( e , Q_d ( p ) ) $ hold, then $ \mathsf{ P }_\lambda ( d + e , p ) $ also holds.
    
\end{lemma}

\begin{proof}

The fact that $ \mathsf{ P }_\lambda ( d , p ) $ and $ \mathsf{ P }_\lambda ( e , Q_d ( p ) ) $ are well-defined follows from Lemma~\ref{lem:Q_d compose}.
Let $ f_1 , f_2 , \dots , f_m $ and $ g_1 , g_2 , \dots , g_m $ satisfy the hypotheses of $ \mathsf{ P }_\lambda ( d + e , p ) $.
It suffices to prove
\begin{align}
K
\leq M_{ Q_{ d + e } ( p ) }^\lambda \left( \frac{ \| g_1 \|_1 }{ \| f_1 \|_1 } , \dots , \frac{ \| g_m \|_1 }{ \| f_m \|_1 } \right). \label{eq:dimension-additivity proof conclude}
\end{align}
The assertion is immediate if $ K = 0 $, so we may assume $ K > 0 $.
Since $ f_i $ and $ g_i $ are integrable for every $ i = 1 , 2 , \dots , m $, Fubini's theorem implies that the slice-integral functions
\begin{align}
    F_i ( \eta ) \coloneqq \int_{ \mathbb{ R }^d } f_i ( \xi , \eta ) \; d \xi , & &
    G_i ( \eta ) \coloneqq \int_{ \mathbb{ R }^d } g_i ( \xi , \eta ) \; d \xi
\end{align}
from $ \mathbb{ R }^e $ to $ [ 0 , \infty ] $ are finite almost everywhere and satisfy
\begin{align}
    \| F_i \|_1 = \| f_i \|_1, & & \| G_i \|_1 = \| g_i \|_1. \label{eq:dimension-additivity proof same}
\end{align}
If necessary, by redefining $ F_i $ and $ G_i $ to be $ 0 $ on the null sets where they are not finite, we may regard them as functions $ F_i , G_i \colon \mathbb{ R }^e \to \mathbb{ R }_{ \geq 0 } $.
Define
\begin{align}
    \widetilde{ S } \coloneqq \{ ( u_1 , \dots , u_m , v_1 , \dots , v_m ) \in ( \mathbb{ R }^d )^m \times ( \mathbb{ R }^e )^m \mid ( u_1 , v_1 , u_2 , v_2 , \dots , u_m , v_m ) \in S \}.
\end{align}
We use the abbreviations
\begin{gather}
\begin{aligned}
    u = ( u_1 , u_2 , \dots , u_m ), & &
v = ( v_1 , v_2 , \dots , v_m ),
\end{aligned}
\\
    L ( u , v ) \coloneqq L ( ( u_1 , v_1 ) , ( u_2 , v_2 ) , \dots , ( u_m , v_m ) ).
\end{gather}
By \eqref{eq:main theorem essential inf}, we have $ K \leq L ( u , v ) $ for almost every $ ( u , v ) \in \widetilde{ S } $.
For $ v = ( v_1 , \dots , v_m ) \in ( \mathbb{ R }^e )^m $, we define
\begin{align}
    S_v \coloneqq \{ u \in ( \mathbb{ R }^d )^m \mid ( u , v ) \in \widetilde{ S } \}.
\end{align}
Then Fubini's theorem gives
\begin{align}
    K \leq \einf_{ u \in S_v } L ( u , v )
\end{align}
for almost every
\begin{align}
    v \in S_F 
    \coloneqq S ( F_1 ) \times S ( F_2 ) \times \dots \times S ( F_m ).
\end{align}
By Fubini's theorem, the function $ \xi \mapsto g_i ( \xi , \eta ) $ is integrable for almost every $ \eta $.
Moreover, since the linear map $ z_\lambda $ is surjective, the inverse image of a null set under $ z_\lambda $ is a null set.
Thus, the function $ \xi \mapsto g_i ( \xi , z_\lambda ( v ) ) $ is integrable for almost every $ v $.
By applying $ \mathsf{ P }_\lambda ( d , p ) $ to the input functions $ \xi \mapsto f_i ( \xi , v_i ) $ and output functions $ \xi \mapsto g_i ( \xi , z_\lambda ( v ) ) $, we have
\begin{align}
    K \leq M_{ Q_d ( p ) }^\lambda \left( \frac{ G_1 ( z_\lambda ( v ) ) }{ F_1 ( v_1 ) } , \dots , \frac{ G_m ( z_\lambda ( v ) ) }{ F_m ( v_m ) } \right)
\end{align}
for almost every $ v \in S_F $.
Since
\begin{align}
    K
    \leq \einf_{ v \in S_F } M_{ Q_d ( p ) }^\lambda \left( \frac{ G_1 ( z_\lambda ( v ) ) }{ F_1 ( v_1 ) } , \dots , \frac{ G_m ( z_\lambda ( v ) ) }{ F_m ( v_m ) } \right)
\end{align}
holds, we get
\begin{align}
    K
    \leq M_{ Q_e ( Q_d ( p ) ) }^\lambda \left( \frac{ \| G_1 \|_1 }{ \| F_1 \|_1 } , \dots , \frac{ \| G_m \|_1 }{ \| F_m \|_1 } \right)
\end{align}
by applying $ \mathsf{ P }_\lambda ( e , Q_d ( p ) ) $.
Lemma~\ref{lem:Q_d compose} and \eqref{eq:dimension-additivity proof same} give \eqref{eq:dimension-additivity proof conclude}, and hence we obtain $ \mathsf{ P }_\lambda ( d + e , p ) $.
\end{proof}

Lemma~\ref{lem:dimension-additivity} yields the following proof of Theorem~\ref{thm:main-BBL}.

\begin{proof}[Proof of Theorem~\ref{thm:main-BBL}]

    We proceed by induction on $ d $.
    The case $ d = 1 $ follows from Section~\ref{subsec:one dimension proof}.

    Assume the assertion holds in dimension $ d $ and consider dimension $ d + 1 $.
    Let $ \lambda $ be a weight and $ - \infty \leq p \leq 1 / ( d + 1 ) $.
    It suffices to prove $ \mathsf{ P }_\lambda ( d + 1 , p ) $.
    By Lemma~\ref{lem:dimension-additivity}, $ \mathsf{ P }_\lambda ( 1 , p ) $ and $ \mathsf{ P }_\lambda ( d , Q_1 ( p ) ) $ are well-defined.
    Section~\ref{subsec:one dimension proof} gives $ \mathsf{ P }_\lambda ( 1 , p ) $.
    Since $ \mathsf{ P }_\lambda ( d , Q_1 ( p ) ) $ also holds by the induction hypothesis, Lemma~\ref{lem:dimension-additivity} yields $ \mathsf{ P }_\lambda ( d + 1 , p ) $.
\end{proof}

\section{Equality examples and optimality of the main theorem}
\label{sec:equality}

In this section, we construct families of functions for which equality in Theorem~\ref{thm:main-BBL} holds in the case of proportional outputs.
We then use these equality examples to establish optimality of the multiplicative constant $ C $ in \eqref{eq:product best} and of the power $ Q_d ( p ) $ on the right-hand side in \eqref{eq:main-BBL-claim}.
For the two-function case of Theorem~\ref{thm:multi input BBL}, Dubuc \cite[Th\'eor\`emes $ B_n $ and 12]{MR444863} characterized the structure of functions in the equality case.
Rather than giving a complete classification of equality cases for Theorem~\ref{thm:main-BBL}, we construct explicit equality examples in the proportional-output case.
These examples also show that both the multiplicative constant $ C $ and the power $ Q_d ( p ) $ in Theorem~\ref{thm:main-BBL} are optimal.

In Section~\ref{subsec:equality product constant best}, we establish equality examples and thereby show that the optimal multiplicative constant is $ C = 1 $.
In Section~\ref{subsec:equality power best}, we prove optimality of the power $ Q_d ( p ) $ when $ m \geq 2 $.

\subsection{Equality examples and optimality of the multiplicative constant}
\label{subsec:equality product constant best}

We construct families of functions attaining equality in Theorem~\ref{thm:main-BBL} by using $ \alpha $-concave functions.
This construction shows that the multiplicative constant $ C = 1 $ is optimal.

\begin{definition}[$ \alpha $-concave function]

    For $ - \infty \leq \alpha \leq \infty $, a function $ \psi \colon \mathbb{ R }^d \to \mathbb{ R }_{ \geq 0 } $ is called $ \alpha $-concave if
    \begin{align}
        M_\alpha^{ \lambda^* } ( \psi ( x_1^* ) , \psi ( x_2^* ) , \dots , \psi ( x_m^* ) )
        \leq \psi ( z_{ \lambda^* } ( x^* ) )
    \end{align}
    holds for any weight $ \lambda^* = ( \lambda_1^* , \lambda_2^* , \dots , \lambda_m^* ) $ and
    \begin{align}
        x^* = ( x_1^* , x_2^* , \dots , x_m^* ) \in S ( \psi )^m. \label{eq:concave function support}
    \end{align}
    
\end{definition}

The following example shows that $ \alpha $-concave functions exist for any $ \alpha $.

\begin{example}
    \label{ex:convex set}

    For any $ - \infty \leq \alpha \leq \infty $, the indicator function $ \psi = \chi_D $ of a convex set $ D \subset \mathbb{ R }^d $ is $ \alpha $-concave.
    
\end{example}

In Example~\ref{ex:convex set}, if the convex set $ D $ has finite positive measure, then $ 0 < \| \psi \|_1 < \infty $.
We use this observation to construct equality examples for Theorem~\ref{thm:main-BBL}.

Next, we prove the following lemma.

\begin{lemma}
    \label{lem:gamma equality}

    Let $ \lambda = ( \lambda_1 , \lambda_2 , \dots , \lambda_m ) $ be a weight.
    For any $ - \infty < p \leq 1 / d $ and $ \gamma_1 , \gamma_2 , \dots , \gamma_m > 0 $, we set
    \begin{align}
    \gamma \coloneqq M_p^\lambda ( \gamma_1 , \gamma_2 , \dots , \gamma_m ). \label{eq:equality gamma definition}
\end{align}
Then
\begin{align}
        M_{ Q_d ( p ) }^\lambda ( \gamma_1^{ 1 - d p } , \gamma_2^{ 1 - d p } , \dots , \gamma_m^{ 1 - d p } )
        = \gamma^{ 1 - d p }. \label{eq:equality proof Q_d}
    \end{align}
    
\end{lemma}

\begin{proof}

If $ p \neq 0 , 1 / d $, then
    \begin{align}
        M_{ Q_d ( p ) }^\lambda ( \gamma_1^{ 1 - d p } , \gamma_2^{ 1 - d p } , \dots , \gamma_m^{ 1 - d p } )
        & = \left( \sum_{ i = 1 }^m \lambda_i \gamma_i^p \right)^{ 1 / Q_d ( p ) } \\
        & = M_p^\lambda ( \gamma_1 , \gamma_2 , \dots , \gamma_m )^{ 1 - d p } \\
        & = \gamma^{ 1 - d p }.
    \end{align}
By Lemmas~\ref{lem:weighted power mean continuous} and \ref{lem:Q_d continuous}, we also obtain \eqref{eq:equality proof Q_d} for $ p = 0 $ and $ p = 1 / d $.
\end{proof}

By using Lemma~\ref{lem:gamma equality}, we now construct equality examples for Theorem~\ref{thm:main-BBL} in the range $ - \infty < p \leq 1 / d $.

\begin{proposition}
\label{prop:equality}

Let $ d $ and $ \lambda $ be as in Theorem~\ref{thm:main-BBL}, and $ p $, $ \gamma_1 , \gamma_2 , \dots , \gamma_m $, and $ \gamma $ be as in Lemma~\ref{lem:gamma equality}.
We suppose that
\begin{align}
\kappa_1 , \kappa_2 , \dots , \kappa_m > 0, & &
\omega \geq 0, & &
\theta = ( \theta_1 , \theta_2 , \dots , \theta_m ) \in ( \mathbb{ R }^d )^m,
\end{align}
and that $ \psi \colon \mathbb{ R }^d \to \mathbb{ R }_{ \geq 0 } $ is an integrable $ ( - p ) $-concave function with $ 0 < \| \psi \|_1 < \infty $.
For $ i = 1 , 2 , \dots , m $, we define the functions $ f_i $ and $ g_i $ by
\begin{align}
f_i ( x ) \coloneqq \kappa_i \psi \left( \frac{ x - \theta_i }{ \gamma_i^p } \right), & &
g_i ( x ) \coloneqq \omega \kappa_i \gamma_i \psi \left( \frac{ x - z_\lambda ( \theta ) }{ \gamma^p } \right). \label{eq:equality function}
\end{align}
The set $ S $ is defined by \eqref{eq:multi input BBL S definition}, and $ L $ and $ K $ are defined by \eqref{eq:main theorem essential inf}.
Then
\begin{align}
K
= M_{ Q_d ( p ) }^\lambda \left( \frac{ \| g_1 \|_1 }{ \| f_1 \|_1 } , \frac{ \| g_2 \|_1 }{ \| f_2 \|_1 } , \dots , \frac{ \| g_m \|_1 }{ \| f_m \|_1 } \right)
= \omega \gamma. \label{eq:equality value}
\end{align}
In particular, equality holds in \eqref{eq:main-BBL-claim}.

\end{proposition}

\begin{proof}

If $ \omega = 0 $, then
\begin{align}
        K
        = M_{ Q_d ( p ) }^\lambda \left( \frac{ \| g_1 \|_1 }{ \| f_1 \|_1 } , \frac{ \| g_2 \|_1 }{ \| f_2 \|_1 } , \dots , \frac{ \| g_m \|_1 }{ \| f_m \|_1 } \right)
        = \omega \gamma
        = 0.
\end{align}
Thus, it suffices to consider $ \omega > 0 $.
We may assume without loss of generality that
\begin{align}
    \kappa_1 = \kappa_2 = \dots = \kappa_m = \omega = 1, & &
    \theta = ( 0 , 0 , \dots , 0 ). \label{eq:equality proof unit}
\end{align}
Indeed, we may replace $ f_i $ and $ g_i $ by
\begin{align}
    \frac{ f_i ( x + \theta_i ) }{ \kappa_i }, & &
    \frac{ g_i ( x + z_\lambda ( \theta ) ) }{ \omega \kappa_i }, \label{eq:equality proof change}
\end{align}
for $ i = 1 , 2 , \dots , m $.
Then, the parameters $ \kappa_i $ and $ \theta_i $ do not affect \eqref{eq:equality value}, and the normalization in $ \omega $ multiplies both sides of \eqref{eq:equality value} by $ 1 / \omega $.

We first prove $ K \geq \gamma $.
It suffices to show $ L ( x ) \geq \gamma $ for any $ x = ( x_1 , \dots , x_m ) \in S $.
By \eqref{eq:equality function}, we have
\begin{align}
        \frac{ g_i ( z_\lambda ( x ) ) }{ f_i ( x_i ) }
        = \gamma_i \psi \left( \frac{ z_\lambda ( x ) }{ \gamma^p } \right) \psi \left( \frac{ x_i }{ \gamma_i^p } \right)^{ - 1 } \label{eq:equality proof function ratio}
\end{align}
for any $ i = 1 , 2 , \dots , m $.
It follows from \eqref{eq:equality gamma definition} that
\begin{align}
        \lambda^* 
        = ( \lambda_1^* , \lambda_2^* , \dots , \lambda_m^* )
        \coloneqq \left( \lambda_1 \left( \frac{ \gamma_1 }{ \gamma } \right)^p , \lambda_2 \left( \frac{ \gamma_2 }{ \gamma } \right)^p , \dots , \lambda_m \left( \frac{ \gamma_m }{ \gamma } \right)^p \right)
\end{align}
is a weight.
Since $ x \in S $, we have $ x_i \in S ( f_i ) $ for any $ i = 1 , 2 , \dots , m $.
Thus, we have
\begin{align}
        x_i^*
        \coloneqq \frac{ x_i }{ \gamma_i^p }
        \in S ( \psi ),
\end{align}
and hence \eqref{eq:concave function support} holds.
Furthermore, we have
\begin{align}
        \frac{ z_\lambda ( x ) }{ \gamma^p }
        = \sum_{ i = 1 }^m \frac{ \lambda_i x_i }{ \gamma^p }
        = \sum_{ i = 1 }^m \lambda_i^* x_i^*
        = z_{ \lambda^* } ( x^* ).
\end{align}
Thus, the equality \eqref{eq:equality proof function ratio} gives
\begin{align}
        L ( x )
        = M_p^\lambda \left( \frac{ \gamma_1 \psi ( z_{ \lambda^* } ( x^* ) ) }{ \psi ( x_1^* ) } , \frac{ \gamma_2 \psi ( z_{ \lambda^* } ( x^* ) ) }{ \psi ( x_2^* ) } , \dots , \frac{ \gamma_m \psi ( z_{ \lambda^* } ( x^* ) ) }{ \psi ( x_m^* ) } \right).
\end{align}
By Fact~\ref{fact:weighted power mean properties} \ref{item:weighted power mean properties homogeneity}, we have
\begin{align}
        L ( x )
        = \psi ( z_{ \lambda^* } ( x^* ) ) M_p^\lambda \left( \frac{ \gamma_1 }{ \psi ( x_1^* ) } , \frac{ \gamma_2 }{ \psi ( x_2^* ) } , \dots , \frac{ \gamma_m }{ \psi ( x_m^* ) } \right).
\end{align}
Then
\begin{align}
        M_p^\lambda \left( \frac{ \gamma_1 }{ \psi ( x_1^* ) } , \frac{ \gamma_2 }{ \psi ( x_2^* ) } , \dots , \frac{ \gamma_m }{ \psi ( x_m^* ) } \right)
        & = \frac{ \gamma }{ M_{ - p }^{ \lambda^* } \left( \psi ( x_1^* ) , \psi ( x_2^* ) , \dots , \psi ( x_m^* ) \right) }
\end{align}
holds by the definition of $ M_p^\lambda $.
Since $ \psi $ is $ ( - p ) $-concave, it follows that $ L ( x ) \geq \gamma $.
Thus, we obtain $ K \geq \gamma $.

Next, we prove
\begin{align}
        M_{ Q_d ( p ) }^\lambda \left( \frac{ \| g_1 \|_1 }{ \| f_1 \|_1 } , \frac{ \| g_2 \|_1 }{ \| f_2 \|_1 } , \dots , \frac{ \| g_m \|_1 }{ \| f_m \|_1 } \right)
        = \gamma. \label{eq:equality proof goal}
\end{align}
Since
\begin{align}
        \| f_i \|_1 = \gamma_i^{ d p } \| \psi \|_1, & &
        \| g_i \|_1 = \gamma_i \gamma^{ d p } \| \psi \|_1, & &
        \frac{ \| g_i \|_1 }{ \| f_i \|_1 }
        = \gamma^{ d p } \gamma_i^{ 1 - d p } \label{eq:equality proof L1 norm}
\end{align}
hold by \eqref{eq:equality function}, we have
\begin{align}
        M_{ Q_d ( p ) }^\lambda \left( \frac{ \| g_1 \|_1 }{ \| f_1 \|_1 } , \frac{ \| g_2 \|_1 }{ \| f_2 \|_1 } , \dots , \frac{ \| g_m \|_1 }{ \| f_m \|_1 } \right)
        & = M_{ Q_d ( p ) }^\lambda ( \gamma^{ d p } \gamma_1^{ 1 - d p } , \gamma^{ d p } \gamma_2^{ 1 - d p } , \dots , \gamma^{ d p } \gamma_m^{ 1 - d p } ) \\
        & = \gamma^{ d p } M_{ Q_d ( p ) }^\lambda ( \gamma_1^{ 1 - d p } , \gamma_2^{ 1 - d p } , \dots , \gamma_m^{ 1 - d p } )
\end{align}
by Fact~\ref{fact:weighted power mean properties} \ref{item:weighted power mean properties homogeneity}.
Thus, Lemma~\ref{lem:gamma equality} yields \eqref{eq:equality proof goal}.

Finally, Theorem~\ref{thm:main-BBL} gives
\begin{align}
        K
        \leq M_{ Q_d ( p ) }^\lambda \left( \frac{ \| g_1 \|_1 }{ \| f_1 \|_1 } , \frac{ \| g_2 \|_1 }{ \| f_2 \|_1 } , \dots , \frac{ \| g_m \|_1 }{ \| f_m \|_1 } \right)
        = \gamma
        \leq K.
\end{align}
Thus, we obtain \eqref{eq:equality value}.
In particular, equality holds in \eqref{eq:main-BBL-claim}.
\end{proof}

We next give equality examples for $ p = - \infty $.

\begin{proposition}
    \label{prop:equality minus infinity}

Let $ p = - \infty $, and $ d $, $ \lambda $, $ \kappa_1 , \kappa_2 , \dots , \kappa_m $, $ \omega $, and $ \theta $ be as in Proposition~\ref{prop:equality}.
We suppose that
\begin{align}
    s_1 , s_2 , \dots , s_m > 0, & &
    s \coloneqq M_1^\lambda ( s_1 , s_2 , \dots , s_m ), \label{eq:equality minus infinity s}
\end{align}
and that $ D \subset \mathbb{ R }^d $ is a convex set of finite positive measure.
For $ i = 1 , 2 , \dots , m $, we define the functions $ f_i $ and $ g_i $ by
\begin{align}
    f_i ( x ) \coloneqq \kappa_i \chi_D \left( \frac{ x - \theta_i }{ s_i } \right), & &
    g_i ( x ) \coloneqq \omega \kappa_i \chi_D \left( \frac{ x - z_\lambda ( \theta ) }{ s } \right). \label{eq:equality minus infinity function}
\end{align}
The set $ S $ is defined by \eqref{eq:multi input BBL S definition}, and $ L $ and $ K $ are defined by \eqref{eq:main theorem essential inf}.
Then
\begin{align}
    K
    = M_{ Q_d ( p ) }^\lambda \left( \frac{ \| g_1 \|_1 }{ \| f_1 \|_1 } , \frac{ \| g_2 \|_1 }{ \| f_2 \|_1 } , \dots , \frac{ \| g_m \|_1 }{ \| f_m \|_1 } \right)
    = \omega. \label{eq:equality minus infinity value}
\end{align}
In particular, equality holds in \eqref{eq:main-BBL-claim}.

\end{proposition}

\begin{proof}

If $ \omega = 0 $, then
\begin{align}
        K
        = M_{ Q_d ( p ) }^\lambda \left( \frac{ \| g_1 \|_1 }{ \| f_1 \|_1 } , \frac{ \| g_2 \|_1 }{ \| f_2 \|_1 } , \dots , \frac{ \| g_m \|_1 }{ \| f_m \|_1 } \right)
        = \omega
        = 0.
\end{align}
Thus, it suffices to consider $ \omega > 0 $.
If necessary, we may assume without loss of generality that
\begin{align}
    \kappa_1 = \kappa_2 = \dots = \kappa_m = \omega = 1, & &
    \theta = ( 0 , 0 , \dots , 0 )
\end{align}
by replacing $ f_i $ and $ g_i $ as in \eqref{eq:equality proof change} for $ i = 1 , 2 , \dots , m $.

We first show that $ K = 1 $.
Let $ x = ( x_1 , x_2 , \dots , x_m ) \in S $.
Then $ x_i / s_i \in D $ for any $ i = 1 , 2 , \dots , m $.
Since
\begin{align}
    \sum_{ i = 1 }^m \frac{ \lambda_i s_i }{ s } = 1
\end{align}
holds by \eqref{eq:equality minus infinity s}, we have
\begin{align}
    \frac{ z_\lambda ( x ) }{ s }
    = \sum_{ i = 1 }^m \frac{ \lambda_i s_i }{ s } \cdot \frac{ x_i }{ s_i } \in D
\end{align}
by convexity of $ D $.
Thus, we get $ g_i ( z_\lambda ( x ) ) = 1 $, and hence
\begin{align}
        L ( x ) = M_{ - \infty }^\lambda ( 1 , 1 , \dots , 1 )
        = 1
\end{align}
holds for almost every $ x \in S $.
Therefore, we obtain $ K = 1 $.

Next, we show
\begin{align}
    M_{ Q_d ( p ) }^\lambda \left( \frac{ \| g_1 \|_1 }{ \| f_1 \|_1 } , \frac{ \| g_2 \|_1 }{ \| f_2 \|_1 } , \dots , \frac{ \| g_m \|_1 }{ \| f_m \|_1 } \right)
        = 1.
\end{align}
By \eqref{eq:equality minus infinity function}, we have
\begin{align}
        \| f_i \|_1 = s_i^d | D | , & &
        \| g_i \|_1 = s^d | D | , & &
        \frac{ \| g_i \|_1 }{ \| f_i \|_1 }
        = \left( \frac{ s }{ s_i } \right)^d, \label{eq:equality minus infinity proof L1 norm}
\end{align}
and
\begin{align}
        Q_d ( p ) = Q_d ( - \infty ) = - \frac{ 1 }{ d }.
\end{align}
Thus, it follows from \eqref{eq:equality minus infinity s} that
\begin{align}
    M_{ Q_d ( p ) }^\lambda \left( \frac{ \| g_1 \|_1 }{ \| f_1 \|_1 } , \frac{ \| g_2 \|_1 }{ \| f_2 \|_1 } , \dots , \frac{ \| g_m \|_1 }{ \| f_m \|_1 } \right)
    = \left( \sum_{ i = 1 }^m \frac{ \lambda_i s_i }{ s } \right)^{ - d }
    = 1.
\end{align}
This proves \eqref{eq:equality minus infinity value}.
Thus, equality holds in \eqref{eq:main-BBL-claim}.
\end{proof}

Propositions~\ref{prop:equality} and \ref{prop:equality minus infinity} show that, for fixed $ d , p , \lambda $, the least constant $ C $ for which \eqref{eq:product best} holds for any family of functions is $ C = 1 $.
Thus, the multiplicative constant $ 1 $ in Theorem~\ref{thm:main-BBL} is optimal.

\subsection{Optimality of the power}
\label{subsec:equality power best}

In this subsection, we prove optimality of the power $ Q_d ( p ) $ when $ m \geq 2 $.
More precisely, we prove the following proposition.

\begin{proposition}[Optimality of the power]
\label{prop:power best}

Let $ m \geq 2 $, and suppose that $ d $, $ \lambda $, and $ p $ are as in Theorem~\ref{thm:main-BBL}.
We fix $ - \infty \leq \widetilde{ Q } < Q_d ( p ) $.
Then there exist integrable functions $ f_1 , f_2 , \dots , f_m , g_1 , g_2 , \dots , g_m \colon \mathbb{ R }^d \to \mathbb{ R }_{ \geq 0 } $ satisfying $ 0 < \| f_i \|_1 < \infty $ for every $ i = 1 , 2 , \dots , m $ such that
\begin{align}
K > M_{ \widetilde{ Q } }^\lambda \left( \frac{ \| g_1 \|_1 }{ \| f_1 \|_1 } , \frac{ \| g_2 \|_1 }{ \| f_2 \|_1 } , \dots , \frac{ \| g_m \|_1 }{ \| f_m \|_1 } \right), \label{eq:power best claim}
\end{align}
with $ S $ defined by \eqref{eq:multi input BBL S definition} and $ L $ and $ K $ defined by \eqref{eq:main theorem essential inf}.

\end{proposition}

\begin{proof}

First, we prove the case of $ - \infty < p < 1 / d $.
In Proposition~\ref{prop:equality}, we choose the parameters $ \kappa_i $, $ \gamma_i $, $ \omega $, and $ \theta $ so that \eqref{eq:equality proof unit} and $ \gamma_1 \neq \gamma_2 $ hold, and define $ f_i $ and $ g_i $ as in that proposition.
Since $ p < 1 / d $, one has $ 1 - d p > 0 $.
Thus, we have
\begin{align}
    \left( \frac{ \| g_1 \|_1 }{ \| f_1 \|_1 } , \frac{ \| g_2 \|_1 }{ \| f_2 \|_1 } , \dots , \frac{ \| g_m \|_1 }{ \| f_m \|_1 } \right)
    = \gamma^{ d p } ( \gamma_1^{ 1 - d p } , \gamma_2^{ 1 - d p } , \dots , \gamma_m^{ 1 - d p } )
\end{align}
by \eqref{eq:equality proof L1 norm}.
This is not a constant vector because $ \gamma_1 \neq \gamma_2 $.
Therefore, Fact~\ref{fact:weighted power mean properties} \ref{item:weighted power mean properties power monotone strict} gives
\begin{align}
    M_{ \widetilde{ Q } }^\lambda \left( \frac{ \| g_1 \|_1 }{ \| f_1 \|_1 } , \frac{ \| g_2 \|_1 }{ \| f_2 \|_1 } , \dots , \frac{ \| g_m \|_1 }{ \| f_m \|_1 } \right)
    < M_{ Q_d ( p ) }^\lambda \left( \frac{ \| g_1 \|_1 }{ \| f_1 \|_1 } , \frac{ \| g_2 \|_1 }{ \| f_2 \|_1 } , \dots , \frac{ \| g_m \|_1 }{ \| f_m \|_1 } \right). \label{eq:power best strict}
\end{align}
Proposition~\ref{prop:equality} then yields \eqref{eq:power best claim}.

Next we prove the case of $ p = - \infty $.
In Proposition~\ref{prop:equality minus infinity}, we choose the parameters $ \kappa_i $, $ s_i $, $ \omega $, and $ \theta $ so that \eqref{eq:equality proof unit} and $ s_1 \neq s_2 $ hold, and define $ f_i $ and $ g_i $ as in that proposition.
We have
\begin{align}
    \left( \frac{ \| g_1 \|_1 }{ \| f_1 \|_1 } , \frac{ \| g_2 \|_1 }{ \| f_2 \|_1 } , \dots , \frac{ \| g_m \|_1 }{ \| f_m \|_1 } \right)
    = s^d \left( \frac{ 1 }{ s_1^d } , \frac{ 1 }{ s_2^d } , \dots , \frac{ 1 }{ s_m^d } \right)
\end{align}
by \eqref{eq:equality minus infinity proof L1 norm}.
This is not a constant vector because $ s_1 \neq s_2 $.
Thus, Fact~\ref{fact:weighted power mean properties} \ref{item:weighted power mean properties power monotone strict} gives \eqref{eq:power best strict}, and hence Proposition~\ref{prop:equality minus infinity} yields \eqref{eq:power best claim}.

Finally, suppose $ p = 1 / d $.
By Fact~\ref{fact:weighted power mean properties} \ref{item:weighted power mean properties power monotone main}, if \eqref{eq:power best claim} holds for some $ 0 < \widetilde{ Q } < \infty $, then
\begin{align}
M_{ \widehat{ Q } }^\lambda \left( \frac{ \| g_1 \|_1 }{ \| f_1 \|_1 } , \frac{ \| g_2 \|_1 }{ \| f_2 \|_1 } , \dots , \frac{ \| g_m \|_1 }{ \| f_m \|_1 } \right)
\leq M_{ \widetilde{ Q } }^\lambda \left( \frac{ \| g_1 \|_1 }{ \| f_1 \|_1 } , \frac{ \| g_2 \|_1 }{ \| f_2 \|_1 } , \dots , \frac{ \| g_m \|_1 }{ \| f_m \|_1 } \right)
< K
\end{align}
for any $ - \infty \leq \widehat{ Q } \leq 0 $.
Thus, it suffices to assume $ 0 < \widetilde{ Q } < \infty $.

We set $ r \coloneqq \widetilde{ Q } / ( 1 + d \widetilde{ Q } ) $.
We choose $ a = ( a_1 , a_2 , \dots , a_m ) \in ( \mathbb{ R }_{ > 0 } )^m $ with $ a_1 \neq a_2 $, and put $ U \coloneqq M_r^\lambda ( a ) $.
Let $ D \subset \mathbb{ R }^d $ be a convex set of finite positive measure.
For $ i = 1 , 2 , \dots , m $, we define
\begin{align}
    f_i ( x ) \coloneqq \chi_D \left( \frac{ x }{ a_i^r } \right), & &
    g_i ( x ) \coloneqq a_i \chi_D \left( \frac{ x }{ U^r } \right).
\end{align}
Since
\begin{align}
    \sum_{ i = 1 }^m \lambda_i \left( \frac{ a_i }{ U } \right)^r
    = 1,
\end{align}
convexity of $ D $ implies $ z_\lambda ( x ) / U^r \in D $ for any $ x \in S $.
Thus, we have $ K = M_{ 1 / d }^\lambda ( a ) $.

On the other hand, we get
\begin{align}
    \| f_i \|_1 = a_i^{ d r } | D | , & &
    \| g_i \|_1 = a_i U^{ d r } | D | , & &
    \frac{ \| g_i \|_1 }{ \| f_i \|_1 } = U^{ d r } a_i^{ 1 - d r }.
\end{align}
Therefore, by Fact~\ref{fact:weighted power mean properties} \ref{item:weighted power mean properties homogeneity}, we have
\begin{align}
    M_{ \widetilde{ Q } }^\lambda \left( \frac{ \| g_1 \|_1 }{ \| f_1 \|_1 } , \frac{ \| g_2 \|_1 }{ \| f_2 \|_1 } , \dots , \frac{ \| g_m \|_1 }{ \| f_m \|_1 } \right)
    = U^{ d r } M_{ \widetilde{ Q } }^\lambda ( a_1^{ 1 - d r } , a_2^{ 1 - d r } , \dots , a_m^{ 1 - d r } ).
\end{align}
Since $ r = \widetilde{ Q } / ( 1 + d \widetilde{ Q } ) $, we have $ \widetilde{ Q } = Q_d ( r ) $.
Thus, Lemma~\ref{lem:gamma equality} gives
\begin{align}
    M_{ \widetilde{ Q } }^\lambda \left( \frac{ \| g_1 \|_1 }{ \| f_1 \|_1 } , \frac{ \| g_2 \|_1 }{ \| f_2 \|_1 } , \dots , \frac{ \| g_m \|_1 }{ \| f_m \|_1 } \right)
    = U
    = M_r^\lambda ( a ).
\end{align}
Since $ r < 1 / d $, Fact~\ref{fact:weighted power mean properties} \ref{item:weighted power mean properties power monotone strict} yields
\begin{align}
M_r^\lambda ( a )
< M_{ 1 / d }^\lambda ( a )
= K.
\end{align}
Thus, we obtain \eqref{eq:power best claim}.
\end{proof}

\section{Extension to general linear coefficients and optimality}
\label{sec:general-linear-coefficients}

In this section, we replace the evaluation point $ z_\lambda ( x ) $ in Theorem~\ref{thm:main-BBL} by the linear combination associated with a tuple
\begin{align}
    \beta = ( \beta_1 , \beta_2 , \dots , \beta_m ) \in ( \mathbb{ R } \setminus \{ 0 \} )^m, \label{eq:beta definition}
\end{align}
namely
\begin{align}
    z_\beta ( x_1 , x_2 , \dots , x_m )
    \coloneqq \sum_{ i = 1 }^m \beta_i x_i .
\end{align}
We do not assume that $ \beta $ is a weight.
That is, we do not assume $ \beta_i > 0 $ or $ \sum_{ i = 1 }^m \beta_i = 1 $.
The resulting inequality follows immediately from Theorem~\ref{thm:main-BBL} by a change of variables, and the optimality statements of Section~\ref{sec:equality} remain valid as well.

Applying Theorem~\ref{thm:main-BBL} after a change of variables in each input function yields the following corollary.
For $ \beta = \lambda $, it agrees with Theorem~\ref{thm:main-BBL}.

\begin{corollary}
    \label{cor:general-linear-coefficients}

    Let $ d $, $ f_1 , f_2 , \dots , f_m $, $ g_1 , g_2 , \dots , g_m $, $ S $, $ \lambda $, and $ p $ be as in Theorem~\ref{thm:main-BBL}, and let $ \beta $ be as in \eqref{eq:beta definition}.
    We define
    \begin{align}
    K_\beta
    \coloneqq \einf_{ x = ( x_1 , x_2 , \dots , x_m ) \in S } M_p^\lambda \left( \frac{ g_1 ( z_\beta ( x ) ) }{ f_1 ( x_1 ) } , \frac{ g_2 ( z_\beta ( x ) ) }{ f_2 ( x_2 ) }, \dots , \frac{ g_m ( z_\beta ( x ) ) }{ f_m ( x_m ) } \right).
    \end{align}
    Then
    \begin{align}
    K_\beta
    \leq M_{ Q_d ( p ) }^\lambda \left( \left( \frac{ \lambda_1 }{ | \beta_1 | } \right)^d \frac{ \| g_1 \|_1 }{ \| f_1 \|_1 } , \left( \frac{ \lambda_2 }{ | \beta_2 | } \right)^d \frac{ \| g_2 \|_1 }{ \| f_2 \|_1 } , \dots , \left( \frac{ \lambda_m }{ | \beta_m | } \right)^d \frac{ \| g_m \|_1 }{ \| f_m \|_1 } \right). \label{eq:general-linear-coefficients claim}
    \end{align}
    
\end{corollary}

\begin{proof}
    For $ i = 1 , 2 , \dots , m $, we define $ f_i^\sharp ( x ) \coloneqq f_i ( \lambda_i x / \beta_i ) $.
    Then
    \begin{align}
        \| f_i^\sharp \|_1
        = \left( \frac{ | \beta_i | }{ \lambda_i } \right)^d \| f_i \|_1.
    \end{align}
    Since the linear map $ x \mapsto \lambda_i x / \beta_i $ is invertible, the essential infimum is preserved under this change of variables.
    Replacing each $ f_i $ by $ f_i^\sharp $ and applying Theorem~\ref{thm:main-BBL} proves the assertion.
\end{proof}

The invertible changes of variables used in the proof also transfer the optimality statements in Propositions~\ref{prop:equality}, \ref{prop:equality minus infinity}, and \ref{prop:power best} to Corollary~\ref{cor:general-linear-coefficients}.
More precisely, the least multiplicative constant $ C $ for which
\begin{align}
K_\beta
\leq C M_{ Q_d ( p ) }^\lambda \left( \left( \frac{ \lambda_1 }{ | \beta_1 | } \right)^d \frac{ \| g_1 \|_1 }{ \| f_1 \|_1 } , \left( \frac{ \lambda_2 }{ | \beta_2 | } \right)^d \frac{ \| g_2 \|_1 }{ \| f_2 \|_1 } , \dots , \left( \frac{ \lambda_m }{ | \beta_m | } \right)^d \frac{ \| g_m \|_1 }{ \| f_m \|_1 } \right)
\end{align}
holds is $ 1 $.
Moreover, the power $ Q_d ( p ) $ on the right-hand side of \eqref{eq:general-linear-coefficients claim} is also optimal when $ m \geq 2 $.
Indeed, for any $ - \infty \leq \widetilde{ Q } < Q_d ( p ) $, the inequality obtained from \eqref{eq:general-linear-coefficients claim} by replacing $ M_{ Q_d ( p ) }^\lambda $ with $ M_{ \widetilde{ Q } }^\lambda $ does not hold with multiplicative constant $ 1 $.

\section*{Acknowledgment}
The author is grateful to Toshihisa Kubo and Hiroshi Tsuji for their helpful comments on the manuscript.
He is also grateful to Hiroshi Tsuji for bringing the paper of Cordero-Erausquin and Maurey to his attention and for pointing out its relevance to the present work.

During the preparation of this work, the author used ChatGPT (OpenAI) as an aid in developing the organization of the manuscript, translating Japanese drafts into English, language editing, and improving the exposition.
The author reviewed and revised all AI-assisted output and independently verified the mathematical content and references.

\printbibliography

\noindent
Takashi Satomi: 

\noindent
School of System Design and Technology, Department of Mathematics and Data Science, Tokyo Denki University, Adachi-ku, Tokyo 120-8551, Japan.

\noindent
RIKEN Interdisciplinary Theoretical and Mathematical Sciences (iTHEMS), Wako, Saitama 351-0198, Japan.

\noindent
E-mail: takashi.satomi@mail.dendai.ac.jp

\end{document}